\documentclass{amsart}
\usepackage{amssymb,mathtools}
\usepackage[british]{babel}
\usepackage{enumitem}
\usepackage[latin1]{inputenc}
\usepackage{url}
\usepackage{hyperref}
\usepackage{tikz-cd}

\newtheorem{theorem}{Theorem}
\numberwithin{theorem}{section}
\newtheorem{corollary}[theorem]{Corollary}
\newtheorem{lemma}[theorem]{Lemma}
\newtheorem{proposition}[theorem]{Proposition}
\theoremstyle{definition}
\newtheorem{definition}[theorem]{Definition}

\newtheorem{example}[theorem]{Example}

\numberwithin{equation}{section}

\newcommand{\rng}{\operatorname{rng}}
\newcommand{\lo}{\mathsf{LO}}
\newcommand{\wo}{\mathsf{WO}}
\newcommand{\po}{\mathsf{PO}}

\newcommand{\To}{\Rightarrow}
\newcommand{\supp}{\operatorname{supp}}
\newcommand{\nf}{=_{\operatorname{NF}}}
\newcommand{\ang}[1]{\langle#1\rangle}
\newcommand{\te}{\textsf{T}}
\newcommand{\se}{\textsf{S}}
\newcommand{\ot}{\textsf{OT}}

\newcommand{\T}{\mathcal{T}}

\title{Bachmann-Howard Derivatives}
\author{Anton Freund}

\address{Fachbereich Mathematik, Technische Universit{\"a}t Darmstadt, Schlossgartenstr.~7, 64289 Darmstadt, Germany}
\email{freund@mathematik.tu-darmstadt.de}

\begin{document}

\begin{abstract}
It is generally accepted that H.~Friedman's gap condition is closely related to iterated collapsing functions from ordinal analysis. But what precisely is the connection? We offer the following answer: In a previous paper we have shown that the gap condition arises from an iterative construction on transformations of partial orders. Here we show that the parallel construction for linear orders yields familiar collapsing functions. The iteration step in the linear case is an instance of a general construction that we call `Bachmann-Howard derivative'. In the present paper, we focus on the unary case, i.\,e., on the gap condition for sequences rather than trees and, correspondingly, on addition-free ordinal notation systems. This is partly for convenience, but it also allows us to clarify a phenomenon that is specific to the unary setting: As shown by van der Meeren, Rathjen and Weiermann, the gap condition on sequences admits two linearizations with rather different properties. We will see that these correspond to different recursive constructions of sequences.
\end{abstract}


\maketitle

\section{Introduction}

There is clearly a parallel between Higman's lemma~\cite{higman52} and Kruskal's theorem~\cite{kruskal60} on embeddings of sequences and trees, respectively. Not least, this parallel is manifest in the fact that both results have an elegant proof by Nash-Williams's minimal bad sequence method~\cite{nash-williams63}. If one wants to make the parallel more precise, it is natural to start with the observation that both sequences and trees are recursive data types. Such a data type can be constructed as the initial fixed point of a suitable transformation. For example, the initial fixed point of $Z\mapsto 1+X\times Z$ is the set~$\textsf{Seq}(X)$ of sequences with entries in~$X$, while the initial fixed point of $X\mapsto\textsf{Seq}(X)$ is the set of ordered finite trees. In~\cite{frw-kruskal} we have studied certain general transformations of partial orders that we call normal $\po$-dilators (alluding to \mbox{Girard's}~\cite{girard-pi2} dilators on linear orders). If $W$ is a normal $\po$-dilator, then its initial fixed point $\mathcal TW$ carries a canonical partial order. In the aforementioned examples, this order coincides with the usual embedding relation from Higman's lemma and Kruskal's theorem (as already observed by Hasegawa~\cite{hasegawa94,hasegawa97}). Thus both these results are instances of a general fact, which we call the uniform Kruskal theorem: If the normal \mbox{$\po$-}dilator~$W$ preserves well partial orders (\textsc{wpo}s), then the so-called Kruskal fixed point~$\mathcal TW$ is a \textsc{wpo} itself. Together with Rathjen and Weiermann, the present author has shown that the uniform Kruskal theorem is equivalent to $\Pi^1_1$-comprehension~\cite{frw-kruskal} (in the setting of reverse mathematics~\cite{simpson09}). This is particularly interesting because it means that the uniform Kruskal theorem exhausts the full strength of the minimal bad sequence method (which has been analyzed by Marcone~\cite{marcone-bad-sequence}), in contrast to Kruskal's original theorem.

Harvey Friedman has introduced a gap condition on embeddings of trees, which leads to a much stronger version of Kruskal's theorem (see the presentation by Simpson~\cite{simpson85}).  Sch\"utte and Simpson~\cite{schuette-simpson} have studied the corresponding condition for embeddings of sequences. It has been observed that the gap condition is related to an interative construction (in particular by Hasegawa~\cite{hasegawa94,hasegawa97}). One way to make this precise has been worked out by the present author~\cite{freund-kruskal-gap}: Given a normal \mbox{$\po$-}dilator~$W$ and a partial order~$X$, one can construct a relativized Kruskal fixed point $\mathcal TW(X)$ that comes with a bijection
\begin{equation}\label{eq:Kruskal-FP}
\iota_X+\kappa_X:X+W(\mathcal TW(X))\to\mathcal TW(X).
\end{equation}
Here addition denotes disjoint union, i.\,e., we have functions $\iota_X:X\to\mathcal TW(X)$ and $\kappa_X:W(\mathcal TW(X))\to\mathcal TW(X)$ such that $\mathcal TW (X)$ is the disjoint union of their images. The order on $\mathcal T W(X)$ is determined by certain inequalities between the values of~$\iota_X$ and $\kappa_X$ (see~\cite{freund-kruskal-gap}). The transformation $X\mapsto\mathcal T W(X)$ can again be equipped with the structure of a normal $\po$-dilator, which we call the Kruskal derivative of~$W$. From now on, the notation $\mathcal T W$ will be reserved for this \mbox{$\po$-}dilator. The single fixed point from~\cite{frw-kruskal} should thus be denoted by  $\mathcal T W(0)$, where~$0$ stands for the empty order. The principle of $\Pi^1_1$-comprehension is still equivalent to the statement that $\mathcal T W$ preserves \textsc{wpo}s if~$W$ does. Now that $\mathcal T W$ is a transformation rather than a single order, we can iterate the construction: Let~$\mathbb T_0$ be the identity on partial orders, considered as a normal $\po$-dilator. Given $\mathbb T_n$, define $\mathbb T_{n+1}^-$ as the Kruskal derivative of~$\textsf{Seq}\circ\mathbb T_n$. Then put $\mathbb T_{n+1}:=\mathbb T_n\circ\mathbb T_{n+1}^-$. In~\cite{freund-kruskal-gap} it is shown that $\mathbb T_n(0)$ is isomorphic to the set of trees with labels in~$\{0,\dots,n-1\}$, ordered according to Friedman's strong gap condition. The analogous (but much simpler) result for a certain collection of sequences is discussed below.

The equivalence between $\Pi^1_1$-comprehension and the uniform Kruskal theorem has been derived from a previous result on the level of linear orders. Each dilator~$D$ (i.\,e., each suitable transformation of well orders) gives rise to a linear order~$\vartheta D$, which we call the Bachmann-Howard fixed point of~$D$ (since it relativizes the Bachmann-Howard ordinal or, more precisely, the notation system from~\cite{rathjen-weiermann-kruskal}). By an earlier result of the author~\cite{freund-thesis,freund-equivalence,freund-categorical,freund-computable}, $\Pi^1_1$-comprehension is equivalent to the statement that $\vartheta D$ is a well order for any dilator~$D$. In the present paper, we will relativize the construction of~$\vartheta D$ to a linear order~$X$. This results in a Bachmann-Howard fixed point $\vartheta D(X)$ over~$X$, which comes with a bijection
\begin{equation*}
\iota_X+\vartheta_X:X+D(\vartheta D(X))\to\vartheta D(X),
\end{equation*}
analogous to~(\ref{eq:Kruskal-FP}). The appropiate inequalities between different values of $\iota_X$ and $\vartheta_X$ will, once again, be a crucial part of the definition. As above, the notation $\vartheta D$ is now reserved for the transformation $X\mapsto\vartheta D(X)$, while the single fixed point from~\cite{freund-computable} should be denoted by~$\vartheta D(0)$. We will see that the transformation~$\vartheta D$ can again be equipped with the structure of a dilator. This dilator~$\vartheta D$ will be called the Bachmann-Howard derivative of~$D$.

Let us discuss the relation between Kruskal and Bachmann-Howard derivatives. By a linearization of a partial order~$P$ by a linear order~$X$ we shall mean an order-reflecting surjection~$f:X\to P$. Here, order-reflecting means that $f(x)\leq_P f(y)$ implies~$x\leq_X y$, which also ensures that~$f$ is injective. Note that this coincides with the usual notion of linearization if we identify~$X$ with its image under~$f$. A linearization of a $\po$-dilator~$W$ by a dilator~$D$ is a natural family of linearizations $D(X)\to W(X)$, one for each linear order~$X$. We will show that any linearization of~$W$ by~$D$ can be transformed into a linearization of the Kruskal derivative~$\mathcal T W$ by the Bachmann-Howard derivative~$\vartheta D$.

We now consider an application of our general constructions. In~\cite{MRW-linear}, van der Meeren, Rathjen and Weiermann study a certain collection of sequences with gap condition, as well as its linearization by iterated collapsing functions. We will give the following systematic reconstruction of these objects:
\begin{enumerate}
\item Let $\textsf{S}^0_0$ and $\textsf{T}^0_0$ be the identity on partial and linear orders, respectively (considered as a normal $\po$-dilator and a dilator in the sense of Girard).
\item Define $\textsf{S}^0_{n+1}$ as the Kruskal derivative of~$\textsf{S}^0_n$, and $\textsf{T}^0_{n+1}$ as the Bachmann-Howard derivative of~$\textsf{T}^0_n$.
\item Write~$1$ for the order with a single element. Then $\textsf{S}^0_n(1)$ coincides with the set of sequences~$\overline{\mathbb S}_n[0]$ from~\cite[Definition~12]{MRW-linear}, ordered by Friedman's strong gap condition. Furthermore, $\textsf{T}^0_n(1)$ coincides with the system~$T_n[0]$ of collapsing functions from~\cite[Definitions~24 to~27]{MRW-linear}.
\end{enumerate}
Now the fact that $T_n[0]\cong\textsf{T}^0_n(1)$ is a linearization of~$\overline{\mathbb S}_n[0]\cong\textsf{S}^0_n(1)$ is immediate by the general result from the previous paragraph, which thus replaces the explicit verification in~\cite[Lemmas 10 and~11]{MRW-linear}. More importantly, our reconstruction clarifies two conceptual points. First, it confirms that gap condition and collapsing functions are closely related, maybe even more closely than expected: they arise by entirely parallel constructions on partial and linear orders, respectively. Secondly, the collapsing functions studied in~\cite{MRW-linear} (and the variant with addition in~\cite{WW-simultaneous-theta}) are supposed to generalize Rathjen's notation system for the Bachmann-Howard ordinal (see~\cite{frw-kruskal}). But do they provide ``the right" generalization? Our reconstruction shows that, in a certain precise sense, the answer is positive.

In order to indicate that our constructions cover a larger range of applications, we sketch two possible modifications. In the first of these we put $\te_0=\te^0_0$ and $\te_{n+1}=\te_n\circ\te_{n+1}^-$, where $\te_{n+1}^-$ is defined as the Bachmann-Howard derivative of~$\textsf{T}_n$ (note the similarity with the clause $\mathbb T_{n+1}=\mathbb T_n\circ\mathbb T_{n+1}^-$ from the discussion of \cite{freund-kruskal-gap} above). It seems that $\textsf{T}_n(1)$ is a linearization of the order~$\mathbb S_n$ from~\cite{MRW-linear}, which is more liberal than the order~$\overline{\mathbb S}_n[0]$. For the second modification, consider a dilator~$S$ that linearizes the $\po$-dilator~$\textsf{Seq}$ from above (e.\,g.~take $S=\omega_2$ as below). Let $D_0$ be the identity on linear orders, define $D_{n+1}^-$ as the Bachmann-Howard derivative of~$S\circ D_n$, and set $D_{n+1}:=D_n\circ D_{n+1}^-$. Since the construction is entirely parallel to the one from~\cite{freund-kruskal-gap} (discussed above), this should yield a linearization of Friedman's gap condition on trees. We expect that the linear orders~$D_n(0)$ are closely related to the iterated collapsing functions with addition that are studied in~\cite{WW-simultaneous-theta}. Details of both modifications remain to be checked. We have sketched them to indicate the potential breadth of our approach.

As observed by van der Meeren, Rathjen and Weiermann~\cite{MRW-linear}, the linearization of $\textsf{S}^0_n(1)\cong\overline{\mathbb S}_n[0]$ by $\textsf{T}^0_n(1)\cong T_n[0]$ does not have maximal order type. It is expected that this phenomenon is specific to the case of sequences, i.\,e., that collapsing functions do exhaust the maximal order type for trees. Our approach provides some justification for this expectation, or at least a systematic explanation. To present the latter, we consider sequences on a more concrete level: The elements of~$\textsf{S}^0_n(X)$ can be represented in the form $\langle i_1,\dots,i_k,x\rangle$ with $i_1=0$ and $x\in X$ (and with some further conditions, see Definition~\ref{def:se-order} below). In view of~(\ref{eq:Kruskal-FP}), the fact that $\textsf{S}^0_{n+1}$ is the Kruskal derivative of~$\textsf{S}^0_n$ is witnessed, amongst others, by injections
\begin{equation*}
\kappa^n_X:\textsf{S}^0_n(\textsf S^0_{n+1}(X))\to\textsf S^0_{n+1}(X).
\end{equation*}
We will see that these functions are given by
\begin{equation*}
\kappa^n_X(\langle i_1,\dots,i_k,\langle j_1,\dots,j_l,x\rangle\,\rangle)=\langle 0,i_1+1,\dots,i_k+1, j_1,\dots,j_l,x\rangle.
\end{equation*}
Note that $j_1$ is the second entry on the right that is equal to zero, which guarantees that $\kappa^n_X$ is injective. To generalize the construction from sequences to trees, think of the last entry of an element $\langle i_1,\dots,i_k,x\rangle\in\textsf{S}^0_n(X)$ as a leaf labelled by~$x$. Indeed, the orders~$\mathbb T_n(X)$ from above (studied in~\cite{freund-kruskal-gap}) consist of trees with labels from~$\{0,\dots,n-1\}\cup X$, where labels from~$X$ are allowed at leaves only. We can now point out a crucial difference between sequences and trees: A sequence different from $\langle\rangle$ (the empty sequence) has a single leaf with known location (the last entry). This means that it suffices to record the leaf label and the rest of the sequence or, more formally, that we have a bijection
\begin{align*}
\{\langle\rangle\}\cup\left(X\times\textsf{S}^0_n(\{\star\})\right)&{}\cong\textsf S^0_n(X),\\
\langle x,\langle i_1,\dots,i_k,\star\rangle\,\rangle&{}\mapsto\langle i_1,\dots,i_k,x\rangle.
\end{align*}
In the case of trees, there are many possible locations for leaf labels, and a similar bijection does not appear to be available. It turns out that we get an alternative construction of the gap condition on sequences (but not on trees). To describe this construction, we recall that the set $\textsf{Seq}(Z)$ of finite sequences in~$Z$ is the initial Kruskal fixed point of the transformation $X\mapsto 1+Z\times X$, i.\,e., that we have
\begin{equation*}
\textsf{Seq}(Z)=\mathcal T W(0)\quad\text{with}\quad W(X)=1+Z\times X.
\end{equation*}
Here $Z\times X$ is the usual product of partial orders, where $(z,x)\leq_{Z\times X}(z',x')$ is equivalent to the conjunction of~$z\leq_Z z'$ and~$x\leq_X x'$. The order $1+Z\times X$ contains a further element but no other strict inequalities. We will see that there is an isomorphism
\begin{equation}\label{eq:gap-seq-alternative}
\textsf{S}^0_{n+1}(1)\cong\textsf{Seq}(\textsf{S}^0_n(1))
\end{equation}
of partial orders, for each~$n\in\mathbb N$. Together with (1) to (3) from above, this yields a second construction of sequences with gap condition in terms of Kruskal derivatives (or fixed points). While the two constructions coincide for partial orders, it turns out that they differ in the linear case. Given a linear order~$Z$, we define $\omega_2(Z)$ as the initial Bachmann-Howard fixed point of the transformation $X\mapsto 1+Z\times X$, i.\,e., we set
\begin{equation}\label{eq:double-exp-fixed-point}
\omega_2(Z)=\vartheta D(0)\quad\text{with}\quad D(X)=1+Z\times X.
\end{equation}
Note that, in the linear case, the single element of~$1$ lies below all elements of~$Z\times X$, while $(z,x)\leq_{Z\times X}(z',x')$ holds if we have $z<_Z z'$ or ($z=z'$ and $x\leq_X x'$). For an ordinal~$\alpha$ and $n\in\mathbb N$, the ordinal $\omega_n^\alpha$ is explained by the recursive clauses $\omega_0^\alpha=\alpha$ and $\omega_{n+1}^\alpha=\omega^{\omega_n^\alpha}$ (towers of exponentials in the sense of ordinal arithmetic). The resulting ordinal~$\omega_2^\alpha$ is isomorphic to the order $\omega_2(1+\alpha)$ from~(\ref{eq:double-exp-fixed-point}), as shown in~\cite{freund-predicative-collapsing}. We now define linear orders~$\sf{OT}^0_n$ by the recursive clauses
\begin{equation}\label{eq:gap-lin-alternative}
\sf{OT}^0_0=1\quad\text{and}\quad\sf{OT}^0_{n+1}=\omega_2(\sf{OT}^0_n).
\end{equation}
As we will see, $\ot^0_n$ coincides with the order $OT_n[0]$ from~\cite[Section~5]{MRW-linear} (up to a typo in the cited reference, see Section~\ref{sect:application-2} below). Important results of the cited paper can now be deduced from general facts about Bachmann-Howard fixed points: First, the parallel between (\ref{eq:gap-seq-alternative}) and~(\ref{eq:gap-lin-alternative}) ensures that~$OT_n[0]$ linearizes~$\textsf{S}^0_n(1)\cong\overline{\mathbb S}_n[0]$. Secondly, iterated applications of the result from~\cite{freund-predicative-collapsing} show that~$OT_n[0]$ has order type $\omega_{2n-1}:=\omega_{2n-1}^1=\omega_{2n}^0$, for any $n>0$. In~\cite{MRW-linear} this fact was established by \mbox{explicit} computations that involve the addition-free Veblen functions, which may seem some\-what ad hoc. Let us recall that $\omega_{2n-1}$ is the maximal order type of the partial order~$\textsf{S}^0_n(1)\cong\overline{\mathbb S}_n[0]$, as shown in~\cite{MRW-linear} (based on results from~\cite{schuette-simpson}).

To summarize, the present paper introduces the general notion of Bachmann-Howard derivative. The latter allows us to give two systematic reconstructions of the gap condition on finite sequences. While the two constructions yield the same result in the case of partial orders, the versions for linear orders lead to two different systems of collapsing functions, both of which have been studied in~\cite{MRW-linear}. The first construction does not realize the maximal order type (in the case of sequences) but seems to be of greater general interest, since it is readily extended from sequences to trees (cf.~\cite{freund-kruskal-gap}). The second construction exploits a property that is specific to sequences, and it realizes the maximal order type.

\section{Theory, part~1: Bachmann-Howard derivatives}\label{sect:BH-derivatives}

In this section we define the notion of Bachmann-Howard derivative, by making the informal explanation from the introduction precise. We then give a proof of existence and uniqueness, including a criterion that is useful for applications.

To recall the definition of dilators, we need some terminology: Write $[X]^{<\omega}$ for the set of finite subsets of a set~$X$. Each $f:X\to Y$ induces a function
\begin{equation*}
[f]^{<\omega}:[X]^{<\omega}\to[Y]^{<\omega}\quad\text{with}\quad [f]^{<\omega}(a)=\{f(x)\,|\,x\in a\}.
\end{equation*}
This yields an endofunctor~$[\cdot]^{<\omega}$ on the category of sets. Given~$a\in[X]^{<\omega}$, we will write $\iota_a:a\hookrightarrow X$ for the inclusion map, provided that~$X$ is clear from the context. Let $\mathsf{LO}$ be the category of linear orders and order embeddings. We will omit the forgetful functor from orders to sets (and thus apply $[\cdot]^{<\omega}$ to orders). Conversely, a subset of an order will often be considered as a suborder. Finally, let us agree that $\rng(f)$ denotes the range (in the sense of image) of~$f$.

\begin{definition}\label{def:lo-dilator}
An $\lo$-dilator consists of a functor $D:\lo\to\lo$ and a natural transformation $\supp:D\To[\cdot]^{<\omega}$ such that the so-called support condition
\begin{equation*}
\supp_Y(\sigma)\subseteq\rng(f)\quad\Rightarrow\quad\sigma\in\rng(D(f))
\end{equation*}
holds for any $\lo$-morphism $f:X\to Y$ and any $\sigma\in D(Y)$. If, in addition, $D(X)$ is well founded for any well order~$X$, then~$(D,\supp^D)$ is called a $\wo$-dilator.
\end{definition}

Note that the converse of the implication in the definition is automatic since~$\supp$ is natural. There is at most one natural transformation $\supp:D\To[\cdot]^{<\omega}$ that satisfies the support condition, since $\supp_X(\sigma)$ is determined as the minimal $a\subseteq X$ with $\sigma\in\rng(D(\iota_a))$. Furthermore, such a natural transformation exists if, any only if, $D$ preserves pullbacks and direct limits, as verified in~\cite[Remark~2.2.2]{freund-thesis}. This means that our definition of \mbox{$\wo$-dilators} coincides with Girard's definition of dilators~\cite{girard-pi2}. We have added the prefix $\wo$ for clarity, since we will later consider variants of dilators on partial orders.

As supports are uniquely determined, we often write $D$ instead of $(D,\supp)$. Sometimes (but not always) we then write $\supp^D$ to refer to~$\supp$. Let us write
\begin{equation*}
\sigma\nf D(\iota_a)(\sigma_0)\quad\text{with $a\in[X]^{<\omega}$ and $\sigma_0\in D(a)$}
\end{equation*}
if the equality holds and we have~$\supp_a(\sigma_0)=a$. This notation allows us to formulate a version of Girard's normal form theorem:

\begin{lemma}\label{lem:normal-forms}
Consider an $\lo$-dilator~$D$ and a linear order~$X$. Each $\sigma\in D(X)$ has a unique normal form $\sigma\nf D(\iota_a)(\sigma_0)$. The latter satisfies $a=\supp_X(\sigma)$.
\end{lemma}
\begin{proof}
For $\sigma=D(\iota_a)(\sigma_0)$ we have $\supp_X(\sigma)=\supp_a(\sigma_0)$ by naturality. Now the existence of a normal form with $a=\supp_X(\sigma)$ follows from the support condition. Uniqueness holds as $\sigma$ determines~$a$ and as the embedding $D(\iota_a)$ is injective.
\end{proof}

As observed by Girard, the normal form theorem entails that dilators are determined (up to natural isomorphism) by their restrictions to the category of finite linear orders. Since this category is essentially small, these restrictions can be represented by sets (rather than proper classes). A formalization in second order arithmetic is available for the countable case (see e.\,g.~\cite{freund-computable} for details). In the present paper we do not work in a specific base theory.

To define the Bachmann-Howard derivative $\vartheta D$ of an $\lo$-dilator~$D$ we must, in particular, specify a transformation $X\mapsto\vartheta D(X)$ of linear orders. The following definition constructs the required orders from syntactic material. A more abstract characterization, which can be easier to handle in applications, will be given later. We point out that the orders~$\vartheta D(X)$ are relativized versions of the Bachmann-Howard fixed points that we have constructed in~\cite{freund-computable}.

\begin{definition}\label{def:fixed-point-syntactic}
Let us consider an $\lo$-dilator~$D$. For each linear order~$X$, we define a set $\vartheta D(X)$ of terms and a binary relation~$<_{\vartheta D(X)}$ on that set by simultaneous recursion. The set $\vartheta D(X)$ is generated by the following clauses:
\begin{enumerate}[label=(\roman*)]
\item For each element $x\in X$ we have a term~$\overline x\in\vartheta D(X)$.
\item Given a finite set $a\subseteq\vartheta D(X)$ that is linearly ordered by~$<_{\vartheta D(X)}$, we add a term $\vartheta\langle a,\sigma\rangle$ for each element~$\sigma\in D(a)$ with $\supp^D_a(\sigma)=a$.
\end{enumerate}
For $s,t\in\vartheta D(X)$ we stipulate that $s<_{\vartheta D(X)}t$ holds if, and only if, one of the following clauses applies:
\begin{enumerate}[label=(\roman*')]
\item We have $s=\overline x$ and $t=\overline y$ with $x<_X y$.
\item The term $s$ is of the form~$\overline x$ while $t$ is of the form $\vartheta\ang{b,\tau}$.
\item We have $s=\vartheta\ang{a,\sigma}$ and $t=\vartheta\ang{b,\tau}$, the restriction of $<_{\vartheta D(X)}$ to $a\cup b$ is a linear order, and one of the following holds:
\begin{itemize}[label={--}]
\item We have $D(\iota_a)(\sigma)<_{D(a\cup b)}D(\iota_b)(\tau)$ for the inclusions $\iota_a:a\hookrightarrow a\cup b$ and $\iota_b:b\hookrightarrow a\cup b$. Furthermore, we have $s'<_{\vartheta D(X)}t$ for all $s'\in a$.
\item We have $s\leq_{\vartheta D(X)}t'$ for some~$t'\in b$ (i.\,e., we have $s<_{\vartheta D(X)}t'$ or $s$~and~$t'$ are the same term).
\end{itemize}
\end{enumerate}
\end{definition}

To justify the recursion in detail, one can argue as follows: In a first step, ignore the reference to $<_{\vartheta D(X)}$ in order to generate a larger set $\vartheta_0D(X)\supseteq\vartheta D(X)$. More precisely, declare that $\vartheta_0D(X)$ contains $\vartheta\ang{a,\sigma}$ whenever $a\subseteq\vartheta_0D(X)$ is finite and $\sigma\in D(a)$ holds with respect to \emph{some} linear order on~$a$. Then define a length function $l:\vartheta_0D(X)\to\mathbb N$ by recursion over terms, stipulating
\begin{equation*}
l(\overline x)=0\quad\text{and}\quad l(\vartheta\ang{a,\sigma})=1+\sum_{r\in a}2\cdot l(r).
\end{equation*}
Finally, decide $r\in\vartheta D(X)$ and $s<_{\vartheta D(X)}t$ by simultaneous recursion on $l(r)$ and $l(s)+l(t)$, respectively. For example, we can decide $r\in\vartheta D(X)$ for $r=\vartheta\ang{a,\sigma}$ as follows: Recursively decide~$a\subseteq\vartheta D(X)$ and compute the restriction of~$<_{\vartheta D(X)}$ to~$a$. If the latter is a linear order, then check $\sigma\in D(a)$ and $\supp^D_a(\sigma)=a$ with respect to it. At various places in Definition~\ref{def:fixed-point-syntactic}, we have required that certain restrictions of $<_{\vartheta D(X)}$ are linear. The purpose was to ensure that~$D:\lo\to\lo$ is only applied to (morphisms of) linear orders. \emph{Ex post}, linearity is automatic by the following lemma, which is proved as in the non-relativized case (see~\cite[Proposition~4.1]{freund-computable}).

\begin{proposition}
The relation $<_{\vartheta D(X)}$ is a linear order on~$\vartheta D(X)$.
\end{proposition}

The syntactic construction of~$\vartheta D(X)$ can be hard to apply, even in simple cases. For this reason, we now develop a more abstract characterization.

\begin{definition}\label{def:BH-fp}
Consider an $\lo$-dilator~$D$ and a linear order~$X$. A Bachmann-Howard fixed point of~$D$ over~$X$ consists of a linear order~$Z$ and functions $\iota:X\to Z$ and $\vartheta:D(Z)\to Z$ with the following properties:
\begin{enumerate}[label=(\roman*)]
\item The function $\iota:X\to Z$ is an order embedding.
\item We have $\iota(x)<_Z\vartheta(\sigma)$ for all $x\in X$ and $\sigma\in D(Z)$.
\item If we have $\sigma<_{D(Z)}\tau$ as well as $z<_Z\vartheta(\tau)$ for all $z\in\supp^D_Z(\sigma)$, then we have $\vartheta(\sigma)<_Z\vartheta(\tau)$. Furthermore, $z<_Z\vartheta(\sigma)$ holds for any $z\in\supp^D_Z(\sigma)$.
\end{enumerate}
\end{definition}

Let us extend the syntactic orders $\vartheta D(X)$ into Bachmann-Howard fixed points:

\begin{definition}\label{def:construct-BH-fixed-point}
To define $\iota_X:X\to\vartheta D(X)$ and $\vartheta_X:D(\vartheta D(X))\to\vartheta D(X)$, we stipulate $\iota_X(x)=\overline x$ and~$\vartheta_X(\sigma)=\vartheta\ang{a,\sigma_0}$ for $\sigma\nf D(\iota_a)(\sigma_0)$.
\end{definition}

As expected, we obtain the following:

\begin{proposition}\label{prop:BH-fixed-point-exists}
The tuple $(\vartheta D(X),\iota_X,\vartheta_X)$ explained by Definitions~\ref{def:fixed-point-syntactic} and~\ref{def:construct-BH-fixed-point} is a Bachmann-Howard fixed point of~$D$ over~$X$.
\end{proposition}
\begin{proof}
It is immediate that conditions~(i) and~(ii) from Definition~\ref{def:BH-fp} are satisfied. Concerning condition~(iii), we point out that $s<_{\vartheta D(X)}\vartheta_X(\sigma)$ is immediate for a terms of the form $s=\overline x$, by clause~(ii') of Definition~\ref{def:fixed-point-syntactic}. The remaining conditions are verified as in the non-relativized case, for which we refer to~\cite[Theorem~4.1]{freund-computable}.
\end{proof}

In the non-relativized case, \cite[Theorem~4.2]{freund-computable} shows that $\vartheta D(0)$ can be embedded into any Bachmann-Howard fixed point of~$D$ over the empty set. Here we establish a stronger categorical property, which will be useful below (see~\cite[Section~3]{freund-kruskal-gap} for parallel constructions in the context of partial orders).

\begin{definition}\label{def:initial-fixed-point}
A Bachmann-Howard fixed point $(Z,\iota,\vartheta)$ of $D$ over~$X$ is called initial if any Bachmann-Howard fixed point $(Z',\iota',\vartheta')$ of~$D$ over~$X$ admits a unique order embedding $f:Z\to Z'$ such that
\begin{equation*}
\begin{tikzcd}[row sep=tiny,column sep=large]
& Z\ar[dd,"f"] & D(Z)\ar[l,"\vartheta",swap]\ar[dd,"D(f)"]\\
X \ar[ru,"\iota"]\ar[rd,"{\iota'}",swap] & & \\
& Z' & D(Z')\ar[l,"{\vartheta'}"]
\end{tikzcd}
\end{equation*}
is a commutative diagram.
\end{definition}

By the usual categorical argument, initial Bachmann-Howard fixed points are unique up to isomorphism. The following yields existence and a useful criterion.

\begin{theorem}\label{thm:initial-characterize}
Consider an $\lo$-dilator~$D$ and a linear order~$X$. For any Bachmann-Howard fixed point $(Z,\iota,\vartheta)$ of~$D$ over~$X$, the following are equivalent:
\begin{enumerate}[label=(\roman*)]
\item We have $Z=\rng(\iota)\cup\rng(\vartheta)$, and there is a function $h:Z\to\mathbb N$ such that $z\in\supp^D_Z(\sigma)$ entails $h(z)<h(\vartheta(\sigma))$, for any $\sigma\in D(Z)$.
\item The Bachmann-Howard fixed point $(Z,\iota,\vartheta)$ is initial.
\end{enumerate}
Furthermore, the Bachmann-Howard fixed point from Proposition~\ref{prop:BH-fixed-point-exists} is initial.
\end{theorem}
\begin{proof}
We first assume~(i) and derive~(ii). Aiming at the latter, we consider an arbitrary Bachmann-Howard fixed point~$(Z',\iota',\vartheta')$ of~$D$ over~$X$. The diagram from Definition~\ref{def:initial-fixed-point} commutes if, and only if, we have
\begin{align*}
f(\iota(x))&=\iota'(x) && \text{for all }x\in X,\\
f(\vartheta(\sigma))&=\vartheta'(D(f\circ\iota_a)(\sigma_0)) && \text{for any }\sigma\nf D(\iota_a)(\sigma_0)\in D(Z).
\end{align*}
The idea is to read these equations as a recursive definition of~$f$, which is possible in view of the following observations: First, $\iota$ and $\vartheta$ are injective by clauses~(i) and~(iii) of Definition~\ref{def:BH-fp} (see also~\cite[Lemma~2.1.7]{freund-thesis}). Secondly, the union $Z=\rng(\iota)\cup\rng(\vartheta)$ is necessarily disjoint, by clause~(ii) of the same definition. Finally, the composition $f\circ\iota_a$ does only depend on values of $f$ on arguments $z\in a=\supp^D_Z(\sigma)$, for which a recursive call is justified due to $h(z)<h(\vartheta(\sigma))$. In view of these observations, there can be at most one embedding~$f:Z\to Z'$ with the required properties. To show that there is one, we first define a length function $l:Z\to\mathbb N$ by the recursive clauses $l(\iota(x))=0$ and $l(\vartheta(\sigma))=1+\sum_{z\in\supp^D_Z(\sigma)}2\cdot l(z)$, which are justified as above. By simultaneous induction on $l(z)$ and $l(z_0)+l(z_1)$ one can now check that $f(z)\in Z'$ is defined and that $z_0<_Zz_1$ entails $f(z_0)<_{Z'}f(z_1)$. The motivation for the simultaneous verification is that $f\circ\iota_a$ needs to be an embedding for $D(f\circ\iota_a)$ to be defined. The only interesting case in the induction concerns an inequality
\begin{equation*}
z_0=\vartheta(\sigma)<_Z\vartheta(\tau)=z_1.
\end{equation*}
Crucially, we must have $z<_Z\vartheta(\tau)$ for all elements $z\in\supp^D_X(\sigma)$, since $\vartheta(\tau)\leq_Z z$ would entail $\vartheta(\tau)<_Z\vartheta(\sigma)$, by clause~(iii) of Definition~\ref{def:BH-fp}. If $\sigma<_{D(Z)}\tau$ fails, we get $\vartheta(\sigma)\leq_Z z'$ for some $z'\in\supp^D_Z(\tau)$, for the same reason.  Writing $\sigma\nf D(\iota_a)(\sigma_0)$ and $\tau\nf D(\iota_b)(\tau_0)$, we can use the induction hypothesis to show that one of the following must hold (write $f\circ\iota_a=f\!\restriction\!(a\cup b)\circ\iota_a'$ with $\iota_a':a\hookrightarrow a\cup b$):
\begin{itemize}[label={--},leftmargin=1.5em]
\item We have $D(f\circ\iota_a)(\sigma_0)<_{D(Z')} D(f\circ\iota_b)(\tau_0)$ and $z<_{Z'}f(\vartheta(\tau))=\vartheta'(D(f\circ\iota_b)(\tau_0))$ for all $z\in[f]^{<\omega}(\supp^D_Z(\sigma))=\supp^D_{Z'}(D(f\circ\iota_a)(\sigma_0))$.
\item We have $\vartheta'(D(f\circ\iota_a)(\sigma_0))\leq_{Z'}z'$ for some $z'\in\supp^D_{Z'}(D(f\circ\iota_b)(\tau_0))$.
\end{itemize}
In either case, clause~(iii) of Definition~\ref{def:BH-fp} yields
\begin{equation*}
f(z_0)=\vartheta'(D(f\circ\iota_a)(\sigma_0))<_{Z'}\vartheta'(D(f\circ\iota_b)(\tau_0))=f(z_1).
\end{equation*}
Next, we use the criterion provided by~(i) to show that the Bachmann-Howard fixed point $(\vartheta D(X),\iota_X,\vartheta_X)$ from Proposition~\ref{prop:BH-fixed-point-exists} is initial. To see that we have
\begin{equation*}
\vartheta D(X)=\rng(\iota_X)\cup\rng(\vartheta_X),
\end{equation*}
it suffices to observe that the same condition $\supp^D_a(\sigma)=a$ appears in clause~(ii) of Definition~\ref{def:fixed-point-syntactic} and in the definition of normal forms (as given in the paragraph before Lemma~\ref{lem:normal-forms}). Let us now define $h:\vartheta D(X)\to\mathbb N$ by recursion over terms, setting $h(\overline x)=0$ and $h(\vartheta\ang{a,\sigma_0})=1+\max\{h(s)\,|\,s\in a\}$. To see that this function has the required property, it suffices to recall that $\sigma\nf D(\iota_a)(\sigma_0)\in D(\vartheta D(X))$ entails~$\supp^D_{\vartheta D(X)}(\sigma)=a$, by Lemma~\ref{lem:normal-forms}. Finally, it is now easy to conclude that (ii) implies~(i), because any initial Bachmann-Howard fixed point must be isomorphic to~$(\vartheta D(X),\iota_X,\vartheta_X)$ (see the proof of~\cite[Theorem~3.5]{freund-kruskal-gap} for details).
\end{proof}

Now that we have established existence and uniqueness up to isomorphism, we will sometimes speak of `the' initial Bachmann-Howard fixed point of $D$ over~$X$ and denote `it' by~$\vartheta D(X)$ (i.\,e., this notation is no longer reserved for the specific term systems from Definition~\ref{def:fixed-point-syntactic}). In~\cite{freund-kruskal-gap} we have introduced a notion of Kruskal derivative for dilators of partial orders. We now define the corresponding notion in the context of linear orders.

\begin{definition}\label{def:BH-deriv}
A Bachmann-Howard derivative of an $\lo$-dilator~$D$ consists of an $\lo$-dilator~$\vartheta D$ and natural families of functions
\begin{equation*}
\iota_X:X\to\vartheta D(X)\quad\text{and}\quad\vartheta_X:D\circ\vartheta D(X)\to\vartheta D(X)
\end{equation*}
such that $(\vartheta D(X),\iota_X,\vartheta_X)$ is an initial Bachmann-Howard fixed point of~$D$ over~$X$, for each linear order~$X$.
\end{definition}

To avoid misunderstanding, we point out that the functions $\vartheta_X$ need not be order embeddings. Hence we do not have a natural transformation $\vartheta:D\circ\vartheta D\Rightarrow\vartheta D$ between endofunctors of linear orders (but the naturality condition is the same).

\begin{proposition}\label{prop:extend-into-derivative}
Assume that, for each linear order~$X$, we are given an initial Bachmann-Howard fixed point $(\vartheta D(X),\iota_X,\vartheta_X)$ of~$D$ over~$X$. There is a unique way to extend this data into a Bachmann-Howard derivative of~$D$.
\end{proposition}
\begin{proof}
We first show that the given map $X\mapsto\vartheta D(X)$ can be uniquely extended into a functor. Given an embedding $f:X\to Y$ of linear orders, it is straightforward to check that $(\vartheta D(Y),\iota_Y\circ f,\vartheta_Y)$ is a Bachmann-Howard fixed point of~$D$ over~$X$. Since $(\vartheta D(X),\iota_X,\vartheta_X)$ is initial, there is a unique embedding~$\vartheta D(f):\vartheta D(X)\to\vartheta D(Y)$ such that the following is a commutative diagram:
\begin{equation*}
\begin{tikzcd}[row sep=tiny,column sep=large]
& \vartheta D(X)\ar[dd,"\vartheta D(f)"] & D(\vartheta D(X))\ar[l,"\vartheta_X",swap]\ar[dd,"D(\vartheta D(f))"]\\
X \ar[ru,"\iota_X"]\ar[rd,"{\iota_Y\circ f}",swap] & & \\
& \vartheta D(Y) & D(\vartheta D(Y))\ar[l,"{\vartheta_Y}"]
\end{tikzcd}
\end{equation*}
The very same diagram (with the triangle written as a square) must commute if the functions $\iota_X$ and $\vartheta_X$ are to be natural in~$X$. Hence there is a unique extension into a functor $\vartheta D:\lo\to\lo$ with the required properties. It remains to consider the extension into an~$\lo$-dilator. First observe that the required functions
\begin{equation*}
\supp^{\vartheta D}_X:\vartheta D(X)\to[X]^{<\omega}
\end{equation*}
are necessarily unique, as naturality and the implication from Definition~\ref{def:lo-dilator} require
\begin{equation*}
\supp^{\vartheta D}_X(\sigma)=\bigcap\{a\in[X]^{<\omega}\,|\,\sigma\in\rng(\vartheta D(\iota_a))\},
\end{equation*}
where $\iota_a:a\hookrightarrow X$ are the inclusions. For existence we use the characterization from part~(i) of Theorem~\ref{thm:initial-characterize}. It allows us to define supports by the recursive clauses
\begin{align*}
\supp^{\vartheta D}_X(\iota_X(x))&=\{x\},\\
\supp^{\vartheta D}_X(\vartheta_X(\sigma))&=\bigcup\{\supp^{\vartheta D}_X(s)\,|\,s\in\supp^D_{\vartheta D(X)}(\sigma)\}.
\end{align*}
Naturality and the implication from Definition~\ref{def:lo-dilator} are checked by induction. For~details we refer to the analogous argument for partial orders (see \cite[Theorem~4.2]{freund-kruskal-gap}).
\end{proof}

Together with the last sentence of Theorem~\ref{thm:initial-characterize}, we get existence:

\begin{corollary}
Any $\lo$-dilator has a Bachmann-Howard derivative.
\end{corollary}

By an isomorphism between $\lo$-dilators $(D,\supp^D)$ and $(E,\supp^E)$ we simply mean a natural isomorphism $\eta:D\To E$ of functors. This is justified because the supports are automatically respected, i.\,e.~we have $\supp^E\circ\eta=\supp^D$. For an isomorphism this is particularly easy to see (cf.~the paragraph before~\cite[Theorem~4.4]{freund-kruskal-gap}). It is also true but more subtle when $\eta$ is merely a natural transformation (see \cite{girard-pi2}). We can now formulate the appropriate uniqueness result:

\begin{proposition}\label{prop:deriv-unique}
If $(\vartheta^0D,\iota^0,\vartheta^0)$ and $(\vartheta^1D,\iota^1,\vartheta^1)$ are two Bachmann-Howard derivatives of the same $\lo$-dilator~$D$, then there is a unique natural isomorphism $\eta:\vartheta^0D\To\vartheta^1D$ such that the diagram
\begin{equation*}
\begin{tikzcd}[row sep=tiny,column sep=large]
& \vartheta^0 D(X)\ar[dd,"\eta_X"] & D(\vartheta^0 D(X))\ar[l,"\vartheta^0_X",swap]\ar[dd,"D(\eta_X)"]\\
X \ar[ru,"\iota^0_X"]\ar[rd,"{\iota^1_X}",swap] & & \\
& \vartheta^1 D(X) & D(\vartheta^1 D(X))\ar[l,"{\vartheta^1_X}"]
\end{tikzcd}
\end{equation*}
commutes for every linear order~$X$.
\end{proposition}
\begin{proof}
Existence and uniqueness of isomorphisms $\eta_X:\vartheta^0D(X)\to\vartheta^1D(X)$ as in the diagram are due to the assumption that $(\vartheta^0D(X),\iota^0_X,\vartheta^0_X)$ and $(\vartheta^1D(X),\iota^1_X,\vartheta^1_X)$ are initial Bachmann-Howard fixed points of~$D$ over~$X$. The non-trivial claim of the proposition is that these isomorphisms are natural in~$X$. This is shown as in the corresponding result for partial orders, for which we refer to~\cite[Theorem~4.4]{freund-kruskal-gap}.
\end{proof}

In the next section we will want to take iterated Bachmann-Howard \mbox{derivatives}. To see that the result is still unique, one should check that the derivatives of isomorphic $\lo$-dilators are isomorphic. This follows from the previous result and the following observation, which is shown as in the partial case (see~\cite[Proposition~4.6]{freund-kruskal-gap}).

\begin{lemma}\label{lem:deriv-it-unique}
If $(\vartheta E,\iota,\vartheta)$ is a Bachmann-Howard derivative of~$E$ and $\eta:D\To E$ is a natural isomorphism, then $(\vartheta E,\iota,\vartheta\bullet\eta)$ is a Bachmann-Howard derivative of~$D$, where $\vartheta\bullet\eta:D\circ\vartheta E\To\vartheta E$ is given by $(\vartheta\bullet\eta)_X=\vartheta_X\circ\eta_{\vartheta E(X)}$.
\end{lemma}

Having established existence and uniqueness, we will speak of `the' Bachmann-Howard derivative of~$D$ and denote `it' by $\vartheta D$. To complete the basic theory of Bachmann-Howard derivatives, we should show that $\vartheta D$ is a $\wo$-dilator (i.\,e., preserves well foundedness) when the same holds for~$D$. Since a particularly short proof of this fact exploits the connection with partial orders, we defer this result until Corollary~\ref{cor:deriv-dilator} below.

\section{Application, part~1: unary collapsing functions}\label{sect:collapsing-functions}

In the introduction we have described a recursive construction of $\lo$-dilators~$\te^0_n$. Starting with the identity functor $\te^0_0:\lo\to\lo$ (which is an $\lo$-dilator with support functions $\supp_X:\te^0_0(X)=X\to[X]^{<\omega}$ given by $\supp_X(x)=\{x\}$), we have defined $\te^0_{n+1}$ as the Bachmann-Howard derivative of~$\te^0_n$. This construction is explained and justified by the results of Section~\ref{sect:BH-derivatives}. Also in the introduction, we have claimed that $\te^0_n(1)$ coincides with the order $T_n[0]$ from~\cite{MRW-linear} (where $1$ is the order with a single element). This claim will be proved in the present section.

Due to the uniqueness results from Section~\ref{sect:BH-derivatives} we can argue `the other way around'. This means that we will not, at first, consider $\te^0_n$ as given by a recursive construction. Instead, we will give an \emph{ad hoc} definition of $\te^0_n$, which extends the definition of the term systems $T_n[0]\cong\te_n(1)$ from~\cite{MRW-linear}. In a second step, we will define transformations $\iota^n:\operatorname{Id}\To\te^0_{n+1}$ and $\vartheta^n:\te^0_n\circ\te^0_{n+1}\To\te^0_{n+1}$ that turn $\te^0_{n+1}$ into a Bachmann-Howard derivative of~$\te^0_n$. By Proposition~\ref{prop:deriv-unique} and Lemma~\ref{lem:deriv-it-unique}, this will entail that our \emph{ad hoc} definition and the recursive construction yield the same result after all (up to natural isomorphism).

\begin{definition}\label{def:term-system-T}
Given a linear order~$X$, we generate a set $\te(X)$ of terms and a function $S:\te(X)\to\mathbb N\cup\{-1\}$ by simultaneous recursion:
\begin{enumerate}[label=(\roman*)]
\item For each $x\in X$, include a term $\overline x\in\te(X)$ with $S(\overline x)=-1$.
\item Given $s\in\te(X)$ and $i\geq\max\{S(s)-1,0\}$, add $\vartheta_is\in\te(X)$ with $S(\vartheta_is)=i$.
\end{enumerate}
For each~$i\in\mathbb N\cup\{-1\}$, let $k_i:\te(X)\to\te(X)$ be given by the recursive clauses
\begin{equation*}
k_i(\overline x)=\overline x\quad\text{and}\quad k_i(\vartheta_js)=\begin{cases}
\vartheta_js & \text{if $j\leq i$},\\
k_i(s) & \text{if $j>i$}.
\end{cases}
\end{equation*}
To define a binary relation $<_{\te(X)}$ on~$\te(X)$, we declare that $s<_{\te(X)}t$ holds if, and only if, one of the following clauses applies:
\begin{enumerate}[label=(\roman*')]
\item We have $s=\overline x$ and $t=\overline y$ with $x<_X y$.
\item The term $s$ is of the form $\overline x$ while $t$ is of the form $\vartheta_jt'$.
\item We have $s=\vartheta_is'$ and $t=\vartheta_jt'$, and one of the following holds:
\begin{itemize}[label={--}]
\item We have $i<j$.
\item We have $i=j$, $s'<_{\te(X)}t'$ and $k_i(s')<_{\te(X)}t$.
\item We have $i=j$ and $s\leq_{\te(X)} k_j(t')$.
\end{itemize}
\end{enumerate}
Concerning the last clause, we clarify that $s\leq_{\te (X)}t$ abbreviates the disjunction of $s<_{\te (X)}t$ and the statement that $s$ and $t$ are the same term.
\end{definition}

To justify the definition of $<_{\te(X)}$ one can employ the function $h:\te(X)\to\mathbb N$ given by $h(\overline x)=0$ and $h(\vartheta_is)=h(s)+1$. An easy induction shows $h(k_i(s))\leq h(s)$. It follows that $s<_{\te(X)}t$ can be decided by recursion on $h(s)+h(t)$.

\begin{lemma}\label{lem:T(X)-linear}
The relation $<_{\te(X)}$ on $\te(X)$ is a linear order.
\end{lemma}
\begin{proof}
For a term $s=\vartheta_is'\in\te (X)$, the $s$-secure subterms of $s$ are defined as follows: The term $s'$ is $s$-secure. And if $\vartheta_jt$ with $j\geq i$ is $s$-secure, then so is~$t$. More intuitively, $t$ is $s$-secure if we have $s'=\vartheta_{j_1}\ldots\vartheta_{j_n}t$ with $j_1,\dots,j_n\geq i$. If $t$ is $s$-secure for $s=\vartheta_is'$, then the following holds: First, $k_i(t)$ is $s$-secure. Secondly, we either have $k_i(t)=k_i(s')$, or $t$ is $k_i(s')$-secure with $k_i(s')=\vartheta_is''$ for the same~$i$. We will establish the following two statements simultaneously by induction on~$s$. To see that the restriction to $s$-secure subterms is necessary, consider the counterexample that would arise from $s=\vartheta_1\vartheta_0\vartheta_1\vartheta_10$ and $t=\vartheta_1\vartheta_10=k_1(t)$.
\begin{enumerate}
\item If $t$ is $s$-secure with $s=\vartheta_is'$, then we cannot have $s\leq_{\te (X)}k_i(t)$.
\item We do not have $s<_{\te (X)}s$.
\end{enumerate}
In the induction step for~(1), we use side induction on~$t$. We may write~$k_i(t)=\vartheta_it'$, because the conclusion is trivial when $k_i(t)$ has a different form. Let us exclude all reasons for which $s\leq_{\te (X)}k_i(t)$ could hold: First, note that $h(k_i(t))\leq h(t)<h(s)$ excludes equality. Secondly, given that $t$ is $s$-secure, the same holds for $t'$. Thus the side induction hypothesis excludes $s\leq_{\te (X)}k_i(t')$. Finally, the only remaining reason would involve~$k_i(s')<_{\te (X)}k_i(t)$. The latter entails that $k_i(s')$ and $k_i(t)$ are different, by part~(2) of the simultaneous induction hypothesis. It follows that $t$ is $k_i(s')$-secure, so that the main induction hypothesis excludes~$k_i(s')<_{\te (X)}k_i(t)$. Concerning the induction step for~(2), we note that only terms of the form~$s=\vartheta_is'$ are interesting. Since the induction hypothesis excludes~$s'<_{\te (X)}s'$, the inequality $s<_{\te (X)}s$ would require~$s\leq_{\te (X)} k_i(s')$. This, however, is excluded by part~(1). Trichotomy and transitivity are readily established by induction on the combined term complexity (e.\,g., on $h(s)+h(t)$ for trichotomy between $s$ and $t$).
\end{proof}

Together with trichotomy, statement~(1) from the previous proof yields:

\begin{corollary}\label{cor:T(X)-monotone}
We have $k_i(s)<_{\te(X)}\vartheta_is$ whenever~$\vartheta_is\in\te(X)$.
\end{corollary}

We will be particularly interested in the following suborders of~$\te(X)$.

\begin{definition}\label{def:T_n}
For $n\in\mathbb N$ and a linear order~$X$, define $\te_n(X)\subseteq\te(X)$ as the suborder of terms that contain indices below~$n$ only. Equivalently, $\te_n(X)$ is generated just as $\te(X)$, but with the additional restriction~$i<n$ in clause~(ii) of Definition~\ref{def:term-system-T}. Furthermore, we define $\te^0_n(X)=\{s\in\te_n(X)\,|\,S(s)\leq 0\}$ as the suborder of terms that have the form $\overline x$ or outer index~$0$. We will also write $<_{\te(X)}$ (or just~$<$) for the restriction of this order to $\te_n(X)$ and $\te^0_n(X)$.
\end{definition}

As mentioned before, we have the following connection:

\begin{corollary}\label{cor:coincide-mrw-linear}
The orders $T_n$ and $T_n[0]$ from~\cite[Section~2.3.3]{MRW-linear} coincide with our orders $\te_n(1)$ and $\te^0_n(1)$, respectively (where $1$ is the order with a single element).
\end{corollary}
\begin{proof}
If we write $1=\{0\}$ and identify $\overline 0\in\te_n(1)$ with $0\in T_n$, the definitions coincide except at one point: Definition~27 of~\cite{MRW-linear} declares that $s>t$ and $\vartheta_is\leq k_i(t)$ imply $\vartheta_is<\vartheta_i t$. In the corresponding clause~(iii') of our Definition~\ref{def:term-system-T}, we have omitted the condition $s>t$, because it turns out to be superfluous: According to \cite[Lemma~9]{MRW-linear} (cf.~also Corollary~\ref{cor:T(X)-monotone} above) we have $k_i(t)<\vartheta_it$. Hence it follows from transitivity that $\vartheta_is\leq k_i(t)$ alone entails $\vartheta_is<\vartheta_i t$.
\end{proof}

We want to show that $\sf T^0_{n+1}$ is a Bachmann-Howard derivative of $\sf T^0_n$. Officially, this claim does only make sense once we have specified a dilator~$\sf T^0_n$ that extends the transformation $X\mapsto\sf T^0_n(X)$ of linear orders. We defer this extension until later, because we want to start with the most interesting part of the construction:

\begin{definition}\label{def:iota-theta}
For each number $n\in\mathbb N$ and any linear order~$X$, we define a function $\sigma^n_X:\te _n\circ\te ^0_{n+1}(X)\to\te _{n+1}(X)$ by the recursive clauses
\begin{equation*}
\sigma^n_X(\overline r)=r\text{ (with $r\in\te ^0_{n+1}(X)$)}\quad\text{and}\quad \sigma^n_X(\vartheta_is)=\vartheta_{i+1}\sigma^n_X(s).
\end{equation*}
Now we define $\vartheta^n_X:\te ^0_n\circ\te ^0_{n+1}(X)\to\te ^0_{n+1}(X)$ by setting $\vartheta^n_X(s)=\vartheta_0\sigma^n_X(s)$. Finally, let $\iota^n_X:X\to\te ^0_{n+1}(X)$ be given by~$\iota^n_X(x)=\overline x$.
\end{definition}

Note that elements $r=\overline x\in\te ^0_{n+1}(X)$ give rise to elements $\overline r=\overline{\overline x}\in\te _n\circ\te ^0_{n+1}(X)$. We have
\begin{equation*}
S(\sigma^n_X(s))=\begin{cases}
S(s)=-1 & \text{if $s$ is of the form $\overline{\overline x}$},\\
S(s)+1 & \text{otherwise}.
\end{cases}
\end{equation*}
For $i\geq 0$, it follows that the condition $i\geq S(s)-1$ from Definition~\ref{def:term-system-T} is equivalent to $i+1\geq S(\sigma^n_X(s))-1$. This justifies the second clause in the definition of~$\sigma^n_X$. To justify the definition of~$\vartheta^n_X$, it suffices to note that $s\in\te ^0_n\circ\te ^0_{n+1}(X)$ entails $S(s)\leq 0$ and hence $0\geq S(\sigma^n_X(s))-1$. To formulate the next result, we need one new piece of notation: For $s\in\te(X)$, the value $k_{-1}(s)\in\te(X)$ is always of the form $\overline x$ with~$x\in X$. We define $\underline k:\te(X)\to X$ by setting
\begin{equation*}
\underline k(s)=x\quad\text{for $k_{-1}(s)=\overline x$}.
\end{equation*}
In particular, $s\in\te ^0_n\circ\te ^0_{n+1}(X)$ yields $\underline k(s)\in\te ^0_{n+1}(X)$. We will see that the following proposition ensures the crucial clause~(iii) of Definition~\ref{def:BH-fp}.

\begin{proposition}\label{prop:T-deriv-prep}
The function $\sigma^n_X:\te _n\circ\te ^0_{n+1}(X)\to\te _{n+1}(X)$ is an order isomorphism, for each $n\in\mathbb N$ and any linear order~$X$. Furthermore, the following holds for all $s,t\in\te ^0_n\circ\te ^0_{n+1}(X)$:
\begin{itemize}[label={--}]
\item If we have $s<t$ and $\underline k(s)<\vartheta^n_X(t)$, then we have~$\vartheta^n_X(s)<\vartheta^n_X(t)$.
\item We have $\underline k(s)<\vartheta^n_X(s)$.
\end{itemize}
\end{proposition}
\begin{proof}
An easy induction over an arbitrary term $r\in\te _{n+1}(X)$ shows that it lies in the range of~$\sigma^n_X$. As preparation for the rest of the proof, one inductively shows
\begin{equation*}
k_{i+1}(\sigma^n_X(s))=\sigma^n_X(k_i(s))
\end{equation*}
for $i\in\mathbb N\cup\{-1\}$. In the crucial case of a term $s=\overline r$, this follows from $k_i(s)=s$ and $\sigma^n_X(s)=r=k_{i+1}(r)$, where the latter relies on $S(r)\leq i+1$ due to $r\in\te^0_{n+1}(X)$. Now an easy induction on $h(s)+h(t)$ shows that $s<t$ implies $\sigma^n_X(s)<\sigma^n_X(t)$. This implication does automatically upgrade to an equivalence, as we are concerned with linear orders. In particular $\sigma^n_X$ is injective, and indeed an order isomorphism. Concerning the claims about~$\vartheta^n_X$, we first note that $\underline k(s)<\vartheta^n_X(t)$ entails
\begin{equation*}
k_0(\sigma^n_X(s))=\sigma^n_X(k_{-1}(s))=\underline k(s)<\vartheta_0\sigma^n_X(t).
\end{equation*}
Since we already know that $s<t$ entails $\sigma^n_X(s)<\sigma^n_X(t)$, we get
\begin{equation*}
\vartheta ^n_X(s)=\vartheta_0\sigma^n_X(s)<\vartheta_0\sigma^n_X(t)=\vartheta ^n_X(t)
\end{equation*}
under the given assumptions. Arguing as before, we see that the remaining claim amounts to $k_0(\sigma^n_X(s))<\vartheta_0\sigma^n_X(s)$, which holds by Corollary~\ref{cor:T(X)-monotone}.
\end{proof}

As promised, we now extend $\te^0_n$ (and in the process also $\te_n$) into an $\lo$-dilator.

\begin{definition}\label{def:T-dilator}
For each embedding $f:X\to Y$ of linear orders, we define a function $\te(f):\te(X)\to\te(Y)$ by the recursive clauses
\begin{equation*}
\te(f)(\overline x)=\overline{f(x)}\quad\text{and}\quad\te(f)(\vartheta_is)=\vartheta_i\te(f)(s).
\end{equation*}
Furthermore, we define functions $\supp^\te_X:\te(X)\to[X]^{<\omega}$ by setting
\begin{equation*}
\supp^\te_X(s)=\{\underline k(s)\}.
\end{equation*}
We will also write $\supp^T_X$ for the restrictions of this function to $\te_n(X)$ and to~$\te^0_n(X)$. By $\te_n(f):\te_n(X)\to\te_n(Y)$ and $\te^0_n(f):\te^0_n(X)\to\te^0_n(Y)$ we denote the restrictions of~$\te(f)$ with (co-)domains as given.
\end{definition}

It is immediate that we have $S(\te(f)(s))=S(s)$, which confirms that the functions $\te(f)$, $\te_n(f)$ and $\te^0_n(f)$ are well-defined with the indicated codomains. We have given recursive definitions because they easily generalize from sequences to more complicated data types (cf.~the treatment of trees in~\cite{freund-kruskal-gap}). In the present case, it may simplify matters if we observe
\begin{equation*}
\te(f)(\vartheta_{j_1}\ldots\vartheta_{j_n}\overline x)=\vartheta_{j_1}\ldots\vartheta_{j_n}\overline{f(x)}\quad\text{and}\quad\supp^\te_X(\vartheta_{j_1}\ldots\vartheta_{j_n}\overline x)=\{x\}.
\end{equation*}
Let us now verify the following:

\begin{proposition}\label{prop:te-dilators}
The previous definition extends $\te_n$ and $\te^0_n$ into $\lo$-dilators.
\end{proposition}
\begin{proof}
A straightforward induction over~$s\in\te(X)$ shows $k_i(\te(f)(s))=\te(f)(k_i(s))$. Given an order embedding $f:X\to Y$, one can now check
\begin{equation*}
s<_{\te(X)}t\quad\Rightarrow\quad\te(f)(s)<_{\te(Y)}\te(f)(t)
\end{equation*}
by induction over~$h(s)+h(t)$ (for $h:\te(X)\to\mathbb N$ as given after Definition~\ref{def:term-system-T}). Two more easy inductions show that~$\te$ respects identity morphisms and compositions. It follows that $\te_n$ and $\te^0_n$ are endofunctors of linear orders. By the definition of $\underline k$ and the first line of this proof, we have $\underline k(T(f)(s))=f(\underline k(s))$. This yields
\begin{equation*}
\supp^\te_Y\circ\te(f)(s)=\{\underline k(T(f)(s))\}=[f]^{<\omega}(\{\underline k(s)\})=[f]^{<\omega}\circ\supp^\te_X(s),
\end{equation*}
so that $\supp^\te$ is natural. To conclude that $\te_n$ is a dilator, we show
\begin{equation*}
\supp^\te_Y(r)\subseteq\rng(f)\quad\Rightarrow\quad r\in\rng(\te_n(f))
\end{equation*}
by induction over~$r\in\te_n(Y)$, still for $f:X\to Y$. In the base case of a term~$r=\overline y$ we observe $\supp^\te_Y(r)=\{\underline k(\overline y)\}=\{y\}$. By the antecedent of our implication we may write $y=f(x)$, which yields $r=\te_n(f)(\overline x)$ as desired. In the step for $r=\vartheta_ir'$ we note that $k_{-1}(r)=k_{-1}(r')$ entails $\supp^\te_Y(r)=\supp^\te_Y(r')$. Given the antecedent of our implication, we can thus invoke the induction hypothesis to get $r'=\te_n(f)(r_0')$ for some $r_0'\in\te_n(X)$. In view of $S(r_0')=S(\te(f)(r_0'))=S(r')$ we may form the term $\vartheta_ir_0'\in\te_n(X)$ to get $r=\te_n(f)(\vartheta_ir_0')\in\rng(\te_n(f))$. In order to deduce the analogous implication for~$\te^0_n$, we need only observe that $\te_n(f)(r_0)=r\in\te^0_n(Y)$ entails $S(r_0)=S(r)\leq 0$ and hence $r_0\in\te^0_n(X)\subseteq\te_n(X)$.
\end{proof}

The following theorem is the main result of this section. We write $\iota^n$ and $\vartheta^n$ for the families of functions $\iota^n_X:X\to\te^0_{n+1}(X)$ and $\vartheta^n_X:\te^0_n\circ\te^0_{n+1}(X)\to\te^0_{n+1}(X)$, which are indexed by the linear order~$X$ (cf.~Definition~\ref{def:iota-theta}).

\begin{theorem}\label{thm:T-n+1-deriv}
The Bachmann-Howard derivative of $\te^0_n$ is given by $(\te^0_{n+1},\iota^n,\vartheta^n)$, for each number~$n\in\mathbb N$.
\end{theorem}
\begin{proof}
A straightforward induction over terms shows that the functions $\sigma^n_X$ from Definition~\ref{def:iota-theta} are natural in~$X$. One can conclude that the same holds for $\iota^n_X$ and~$\vartheta^n_X$. In view of Definition~\ref{def:BH-deriv}, it remains to show that $(\te^0_{n+1}(X),\iota^n_X,\vartheta^n_X)$ is an initial Bachmann-Howard fixed point of~$\te^0_n$ over~$X$, whenever~$X$ is a linear~order. Clauses~(i) and~(ii) of Definition~\ref{def:BH-fp} are immediate by our constructions. Clause~(iii) holds by Proposition~\ref{prop:T-deriv-prep} and the definition of~$\supp^\te$. To complete the proof, we verify the criteria from part~(i) of Theorem~\ref{thm:initial-characterize}. The first criterion demands
\begin{equation*}
\te^0_{n+1}(X)=\rng(\iota^n_X)\cup\rng(\vartheta^n_X).
\end{equation*}
To see that this holds, consider a term $\vartheta_0s\in\te^0_{n+1}(X)$. We note that $s\in\te_{n+1}(X)$ must satisfy $S(s)\leq 1$, due to Definition~\ref{def:term-system-T}. From Proposition~\ref{prop:T-deriv-prep} we know that $\sigma^n_X$ is surjective, which yields $s=\sigma^n_X(s')$ for some $s'\in\te_n\circ\te^0_{n+1}(X)$. By the paragraph after Definition~\ref{def:iota-theta} we get $S(s')=-1$ or $S(s')=S(s)-1\leq 0$, which means that we even have $s'\in\te^0_n\circ\te^0_{n+1}(X)$. We now see
\begin{equation*}
\vartheta_0s=\vartheta_0\sigma^n_X(s')=\vartheta^n_X(s')\in\rng(\vartheta^n_X)\quad\text{with}\quad\vartheta^n_X:\te^0_n\circ\te^0_{n+1}(X)\to\te^0_{n+1}(X),
\end{equation*}
as required. The second criterion requires a function $h:\te^0_{n+1}(X)\to\mathbb N$ with
\begin{equation*}
h(r)<h(\vartheta^n_X(s))\quad\text{for any }s\in\te^0_n\circ\te^0_{n+1}(X)\text{ and }r\in\supp^\te_{\te^0_{n+1}(X)}(s).
\end{equation*}
We show that this holds for (the restriction of) the function $h:\te(X)\to\mathbb N$ from the paragraph after Definition~\ref{def:term-system-T}. By our definition of supports, the only possibility is $r=\underline k(s)$. By induction over $s\in\te_n\circ\te^0_{n+1}(X)$ one can verify $h(\underline k(s))\leq h(\sigma^n_X(s))$. In view of $h(\vartheta^n_X(s))=h(\vartheta_0\sigma^n_X(s))=h(\sigma^n_X(s))+1$ this yields the claim.
\end{proof}

The $\lo$-dilator $\te^0_n$ has been defined in two different ways: First, we have constructed $\te^0_n$ in terms of iterated Bachmann-Howard derivatives, according to steps~(1) and~(2) from the introduction. Secondly, we have given an \emph{ad hoc} construction of~$\te^0_n$ in Definitions~\ref{def:term-system-T},~\ref{def:T_n} and~\ref{def:T-dilator}. The results of the two constructions coincide by Theorem~\ref{thm:T-n+1-deriv}, as explained in the first two paragraphs of this section. Let us point out that there is, nevertheless, an interesting difference: For the \emph{ad hoc} definition of $\te^0_n$, we needed to define $\te_n$ as an auxiliary construct (which also appears in many inductive verifications). In contrast, the recursive construction via Bachmann-Howard derivatives yields $\te^0_n$ directly.

\section{Theory, part~2: connecting with Kruskal derivatives}\label{sect:linear-partial}

In Section~\ref{sect:BH-derivatives} we have introduced the Bachmann-Howard derivative~$\vartheta D$ of an $\lo$-dilator~$D$. A parallel construction on the level of partial orders was previously studied in~\cite{freund-kruskal-gap}: For each suitable dilator~$W$ on partial orders, it yields the so-called Kruskal derivative~$\mathcal T W$. In the present section we establish fundamental connections between Bachmann-Howard and Kruskal derivatives, i.\,e., between the linear and the partial case.

Important notions from~\cite{freund-kruskal-gap} will be recalled informally, but the reader may need to consult the cited reference for precise definitions. A function $f:X\to Y$ between partial orders is a quasi embedding if it is order reflecting, i.\,e., if \mbox{$f(x)\leq_Y f(x')$} implies $x\leq_X x'$. We consider the category~$\po$ of partial orders with the quasi embeddings as morphisms. A $\po$-dilator is a functor $W:\po\to\po$ that satisfies certain conditions, in particular a support condition as in Definition~\ref{def:lo-dilator} above (see~\cite[Definition~2.1]{freund-kruskal-gap} for details). We call $W$ a $\mathsf{WPO}$-dilator if, in addition, $W(X)$ is a well partial order whenever the same holds for~$X$. By $W\!\restriction\!\lo:\lo\to\po$ we denote the restriction of a $\po$-dilator~$W$ to the category of linear orders. Also, we sometimes consider an $\lo$-dilator~$D$ as a functor $D:\lo\to\po$, i.\,e., we implicitly compose it with the inclusion~$\lo\hookrightarrow\po$. We can then consider $\nu$ as in the following:

\begin{definition}
By a quasi embedding of an $\lo$-dilator~$D$ into a $\po$-dilator~$W$ we mean a natural transformation~$\nu:D\To W\!\restriction\!\lo$.
\end{definition}

Note that $\nu$ consists of a quasi embedding $\nu_X:D(X)\to W(X)$ for each linear order~$X$. If a $\po$-dilator $W$ satisfies a certain normality condition, then it has an essentially unique Kruskal derivative~$\mathcal T W$ (see Definition~2.3 and Section~4 of~\cite{freund-kruskal-gap}). The latter comes with natural families of functions
\begin{equation*}
\iota^W_X:X\to\T W(X)\quad\text{and}\quad\kappa^W_X:W\circ\T W(X)\to\T W(X),
\end{equation*}
indexed by the partial order~$X$. From Section~\ref{sect:BH-derivatives} above we recall that the Bachmann-Howard derivative~$\vartheta D$ of an $\lo$-dilator~$D$ comes with functions
\begin{equation*}
\iota^D_X:X\to\vartheta D(X)\quad\text{and}\quad\vartheta^D_X:D\circ\vartheta D(X)\to\vartheta D(X),
\end{equation*}
where~$X$ does now range over linear orders. Together with Theorem~\ref{thm:linearization-preserved} below, the following result provides the connection between Kruskal and Bachmann-Howard derivatives. The theorem extends \cite[Theorem~4.5]{frw-kruskal}, in which $\nu^+_X$ is only constructed for the empty order~$X=0$. Even though the main idea remains the same, we provide full details, as the setting and notation in~\cite{frw-kruskal} are somewhat different.

\begin{theorem}\label{thm:quasi-embeddings-derivatives}
Let $\nu:D\To W\!\restriction\!\lo$ be a quasi embedding of an $\lo$-dilator~$D$ into a normal $\po$-dilator~$W$. There is a unique quasi embedding $\nu^+:\vartheta D\To\mathcal T W\!\restriction\!\lo$ such that the diagram
\begin{equation*}
\begin{tikzcd}[row sep=tiny,column sep=large]
& \vartheta D(X)\ar[dd,"{\nu^+_X}"] & D\circ\vartheta D(X)\ar[l,"{\vartheta^D_X}",swap]\ar[dd,"{W(\nu^+_X)\circ\nu_{\vartheta D(X)}}"]\\
X \ar[ru,"{\iota^D_X}"]\ar[rd,"{\iota^W_X}",swap] & & \\
& \T W(X) & W\circ\T W(X)\ar[l,"{\kappa^W_X}"]
\end{tikzcd}
\end{equation*}
commutes for each linear order~$X$.
\end{theorem}
\begin{proof}
By Lemma~\ref{lem:normal-forms}, each $\sigma\in D\circ\vartheta D(X)$ has a normal form $\sigma\nf D(\iota_a)(\sigma_0)$ with $a\subseteq\vartheta D(X)$ and $\sigma_0\in D(a)$. Note that $\nu_{\vartheta D(X)}\circ D(\iota_a)=W(\iota_a)\circ\nu_a$ holds since $\nu$ is natural. Hence the diagram in the theorem commutes if, and only if, we have
\begin{align*}
\nu^+_X(\iota^D_X(x))&=\iota^W_X(x) &&\text{for }x\in X,\\
\nu^+_X(\vartheta^D_X(\sigma))&=\kappa^W_X\circ W(\nu^+_X\circ\iota_a)\circ\nu_a(\sigma_0) &&\text{for }\sigma\nf D(\iota_a)(\sigma_0).
\end{align*}
The idea is to read these equations as recursive clauses, which is justified as follows: According to Definition~\ref{def:BH-deriv}, the tuple $(\vartheta D(X),\iota^D_X,\vartheta^D_X)$ is an initial Bachmann-Howard fixed point of~$D$ over~$X$. By (the proof of) Theorem~\ref{thm:initial-characterize}, it follows that the functions $\iota^D_X$ and $\vartheta^D_X$ are injective, and that $\vartheta D(X)$ is the disjoint union of their ranges. Furthermore, the same theorem yields a function $h:\vartheta D(X)\to\mathbb N$ such that
\begin{equation*}
s\in\supp^D_{\vartheta D(X)}(\sigma)\quad\To\quad h(s)<h(\vartheta^D_X(\sigma))
\end{equation*}
holds for any element $\sigma\in D\circ\vartheta D(X)$. Here $\supp^D$ is the support that comes with the $\lo$-dilator~$D$ (see Definition~\ref{def:lo-dilator} and the discussion that follows it). Now recall that $\sigma\nf D(\iota_a)(\sigma_0)$ entails $\supp^D_{\vartheta D(X)}(\sigma)=a$, by Lemma~\ref{lem:normal-forms}. This means that the clauses above define $\nu^+_X(s)$ by recursion over~$h(s)$. More precisely, a straightforward induction on~$h(s)$ shows that the value $\nu^+_X(s)$ is uniquely determined, i.\,e., that there is at most one quasi embedding $\nu^+_X$ such that the diagram in the theorem commutes. The proof of existence is somewhat more subtle, since we must simultaneously show that $\nu^+_X$ is a quasi embedding, in order to ensure that $W(\nu^+_X\circ\iota_a)$ is defined. Let us define $l:\vartheta D(X)\to\mathbb N$ by stipulating $l(\iota^D_X(x))=0$ and
\begin{equation*}
l(\vartheta^D_X(\sigma))=1+\textstyle\sum_{s\in a}2\cdot l(s)\quad\text{for}\quad\sigma\nf D(\iota_a)(\sigma_0),
\end{equation*}
which is itself a recursion based on~$h$. By simultaneous induction on the values $l(r)$ and $l(s)+l(t)$ one can now show that $\nu^+_X(r)\in\T W(X)$ is defined and that we have
\begin{equation*}
\nu^+_X(s)\leq_{\T W(X)}\nu^+_X(t)\quad\To\quad s\leq_{\vartheta D(X)}t.
\end{equation*}
To establish this implication, we distinguish cases according to the forms of~$s$ and~$t$. First assume $s=\iota^D_X(x)$ and $t=\iota^D_X(y)$. Then the antecedent of our implication amounts to $\iota^W_X(x)\leq_{\T W(X)}\iota^W_X(y)$. According to~\cite[Definition~3.1]{freund-kruskal-gap} we get $x\leq_X y$. Now $s\leq_{\vartheta D(X)}t$ follows by clause~(i) of Definition~\ref{def:BH-fp} above. By the same definition, we always have $s\leq_{\vartheta D(X)}t$ for terms of the form $s=\iota^D_X(x)$ and $t=\vartheta^D_X(\tau)$. For terms $s=\vartheta^D_X(\sigma)$ and $t=\iota^D_X(y)$, say with $\sigma\nf D(\iota_a)(\sigma_0)$, we have
\begin{equation*}
\nu^+_X(s)=\kappa^W_X\circ W(\nu^+_X\circ\iota_a)\circ\nu_a(\sigma_0)\not\leq_{\T W(X)}\iota^W_X(x)=\nu^+_X(t),
\end{equation*}
again by~\cite[Definition~3.1]{freund-kruskal-gap}. It remains to compare terms $s=\vartheta^D_X(\sigma)$ and $t=\vartheta^D_X(\tau)$. We write
$\sigma\nf D(\iota_a)(\sigma_0)$ and $\tau\nf D(\iota_b)(\tau_0)$ and assume
\begin{equation*}
\kappa^W_X\circ W(\nu^+_X\circ\iota_a)\circ\nu_a(\sigma_0)=\nu^+_X(s)\leq_{\T W(X)}\nu^+_X(t)=\kappa^W_X\circ W(\nu^+_X\circ\iota_b)\circ\nu_b(\tau_0).
\end{equation*}
According to~\cite[Definition~3.1]{freund-kruskal-gap}, this inequality can hold for two different reasons. In the first case we have $\nu^+_X(s)\leq_{\T W(X)}t'$ for some element
\begin{equation*}
t'\in\supp^W_{\T W(X)}\left(W(\nu^+_X\circ\iota_b)\circ\nu_b(\tau_0)\right)=[\nu^+_X\circ\iota_b]^{<\omega}\left(\supp^W_b(\nu_b(\tau_0))\right).
\end{equation*}
By~\cite[Lemmas~4.2 and~4.4]{frw-kruskal} any quasi embedding of an $\lo$-dilator into a $\po$-dilator respects supports. This means that we have $\supp^W_b(\nu_b(\tau_0))=\supp^D_b(\tau_0)=b$, where the last equality comes from the normal form condition (see the paragraph before Definition~\ref{lem:normal-forms}). We can thus write $t'=\nu^+_X(t_0')$ for some $t_0'\in b$. The latter entails that we have $l(t_0')<l(t)$, so that we get $s\leq_{\vartheta D(X)}t_0'$ by induction hypothesis. By Lemma~\ref{lem:normal-forms} we have $b=\supp^D_{\vartheta D(X)}(\tau)$, so that clause~(iii) of Definition~\ref{def:BH-fp} allows us to conclude~$t_0'<_{\vartheta D(X)}\vartheta^D_X(\tau)=t$. Now transitivity yields $s\leq_{\vartheta D(X)}t$, as desired (in fact the inequality is strict in this case). In the remaining case, the above inequality $\nu^+_X(s)\leq_{\T W(X)}\nu^+_X(t)$ holds because we have
\begin{equation*}
W(\nu^+_X\circ\iota_a)\circ\nu_a(\sigma_0)\leq_{W\circ\T W(X)}W(\nu^+_X\circ\iota_b)\circ\nu_b(\tau_0).
\end{equation*}
Let us factor~$\iota_a=\iota_{a\cup b}\circ\iota_a'$ and $\iota_b=\iota_{a\cup b}\circ\iota_b'$ with $\iota_{a\cup b}:a\cup b\hookrightarrow\vartheta D(X)$. The induction hypothesis ensures that $\nu^+_X\circ\iota_{a\cup b}$ is a quasi embedding, which allows us to form the quasi embedding $W(\nu^+_X\circ\iota_{a\cup b})$. The previous inequality thus entails
\begin{equation}\label{eq:ineq-for-normal}
W(\iota_a')\circ\nu_a(\sigma_0)\leq_{W(a\cup b)}W(\iota_b')\circ\nu_b(\tau_0).
\end{equation}
Due to the naturality of~$\nu$, we get $\nu_{a\cup b}\circ D(\iota_a')(\sigma_0)\leq_{W(a\cup b)}\nu_{a\cup b}\circ D(\iota_b')(\tau_0)$. Since $\nu_{a\cup b}$ reflects the order while $D(\iota_{a\cup b})$ preserves it, this entails
\begin{equation*}
\sigma=D(\iota_a)(\sigma_0)=D(\iota_{a\cup b})\circ D(\iota_a')(\sigma_0)\leq_{D\circ\vartheta D(\sigma)}D(\iota_{a\cup b})\circ D(\iota_b')(\tau_0)=\tau.
\end{equation*}
In order to conclude $s=\vartheta^D_X(\sigma)\leq_{\vartheta D(X)}\vartheta^D_X(\tau)=t$ by clause~(iii) of Definition~\ref{def:BH-fp}, it remains to show that we have
\begin{equation*}
r<_{\vartheta D(X)}\vartheta^D_X(\tau)\quad\text{for all }r\in\supp^D_{\vartheta D(X)}(\sigma)=a.
\end{equation*}
Analogous to the above, we get
\begin{equation*}
a=\supp^D_a(\sigma_0)=[\iota_a']^{<\omega}\left(\supp^W_a(\nu_a(\sigma_0))\right)=\supp^W_{a\cup b}\left(W(\iota_a')\circ\nu_a(\sigma_0)\right),
\end{equation*}
as well as $b=\supp^W_{a\cup b}\left(W(\iota_b')\circ\nu_b(\tau_0)\right)$. Due to inequality~(\ref{eq:ineq-for-normal}) and the assumption that~$W$ is normal (cf.~\cite[Definition~2.3]{freund-kruskal-gap}), it follows that any~$r\in a$ admits an $r'\in b$ with $r\leq_{\vartheta D(X)}r'$ (note that the inequality holds in~$\vartheta D(X)$ because $a\cup b$ is a suborder of the latter). Once again, clause~(iii) of Definition~\ref{def:BH-fp} yields $r'<_{\vartheta D(X)}\vartheta^D_X(\tau)$. Now transitivity allows us to conclude $r<_{\vartheta D(X)}\vartheta^D_X(\tau)$, as needed. This completes the simultaneous proof that~$\nu^+_X$ is well defined and a quasi embedding. It remains to show naturality. Given a quasi embedding $f:X\to Y$, we prove
\begin{equation*}
\nu^+_Y\circ\vartheta D(f)(s)=\T W(f)\circ\nu^+_X(s)
\end{equation*}
by induction over~$h(s)$, for $h:\vartheta D(X)\to\mathbb N$ as above. The crucial case concerns a term $s=\vartheta^D_X(\sigma)$, say with $\sigma\nf D(\iota_a)(\sigma_0)$. Using the naturality of $\vartheta^D$, we get
\begin{equation*}
\vartheta D(f)(s)=\vartheta D(f)\circ\vartheta^D_X\circ D(\iota_a)(\sigma_0)=\vartheta^D_Y\circ D(\vartheta D(f)\circ\iota_a)(\sigma_0).
\end{equation*}
In order to apply $\nu^+_Y$ to the expression on the right side, we need to determine the normal form of~$D(\vartheta D(f)\circ\iota_a)(\sigma_0)$. Consider the restriction $\vartheta D(f)\!\restriction\!a:a\to b$ with codomain $b:=[\vartheta D(f)]^{<\omega}(a)$. In view of $\vartheta D(f)\circ\iota_a=\iota_b\circ(\vartheta D(f)\!\restriction\!a)$ we get
\begin{equation*}
D(\vartheta D(f)\circ\iota_a)(\sigma_0)=D(\iota_b)(\sigma_1)\quad\text{for}\quad\sigma_1:=D(\vartheta D(f)\!\restriction\!a)(\sigma_0)\in D(b).
\end{equation*}
This expression is in normal form, since the naturality of supports yields
\begin{equation*}
\supp^D_b(\sigma_1)=[\vartheta D(f)\!\restriction\!a]^{<\omega}\left(\supp^D_a(\sigma_0)\right)=[\vartheta D(f)\!\restriction\!a]^{<\omega}(a)=b,
\end{equation*}
where $\supp^D_a(\sigma_0)=a$ comes from the normal form condition for $\sigma\nf D(\iota_a)(\sigma_0)$. By the recursive definition of~$\nu^+_Y$, we now obtain
\begin{equation*}
\nu^+_Y\circ\vartheta D(f)(s)=\kappa^W_Y\circ W(\nu^+_Y\circ\iota_b)\circ\nu_b(\sigma_1).
\end{equation*}
The naturality of $\nu$ yields $\nu_b(\sigma_1)=W(\vartheta D(f)\!\restriction\!a)\circ\nu_a(\sigma_0)$ and hence
\begin{equation*}
W(\iota_b)\circ\nu_b(\sigma_1)=W\left(\iota_b\circ(\vartheta D(f)\!\restriction\!a)\right)\circ\nu_a(\sigma_0)=W\left(\vartheta D(f)\circ\iota_a\right)\circ\nu_a(\sigma_0).
\end{equation*}
Since $r\in a$ entails $h(r)<h(s)$, we have $\nu^+_Y\circ\vartheta D(f)\circ\iota_a=\T W(f)\circ\nu^+_X\circ\iota_a$ by induction hypothesis. Putting things together, we can finally conclude
\begin{alignat*}{3}
\nu^+_Y\circ\vartheta D(f)(s)&=\kappa^W_Y\circ W(\T W(f)\circ\nu^+_X\circ\iota_a)\circ\nu_a(\sigma_0)&&={}\\
{}&=\T W(f)\circ\kappa^W_X\circ W(\nu^+_X\circ\iota_a)\circ\nu_a(\sigma_0)&&=\T W(f)\circ\nu^+_X(s),
\end{alignat*}
as needed for the inductive proof that $\nu^+$ is natural.
\end{proof}

We now deduce the result that was promised at the end of Section~\ref{sect:BH-derivatives}. Our proof is somewhat indirect but nevertheless instructive, as it connects several fundamental facts. Even though we do not formalize the present paper in a specific meta theory, we point out that the following argument uses $\Pi^1_1$-comprehension, in the form of the minimal bad sequence lemma. This is unavoidable by the results of~\cite{freund-equivalence,freund-computable}.

\begin{corollary}\label{cor:deriv-dilator}
If $D$ is a $\wo$-dilator (i.\,e., preserves well foundedness), then so is its Bachmann-Howard derivative~$\vartheta D$.
\end{corollary}
\begin{proof}
That $\vartheta D$ is an $\lo$-dilator is guaranteed by Definition~\ref{def:BH-deriv} (but see also Proposition~\ref{prop:extend-into-derivative}). It remains to show that $\vartheta D(X)$ is a well order when the same holds for~$X$. According to~\cite[Section~5]{frw-kruskal} there is a quasi embedding~$\nu:D\To W_D\!\restriction\!\lo$ into a normal $\mathsf{WPO}$-dilator~$W_D$. By the previous theorem we get a quasi embedding
\begin{equation*}
\nu^+_X:\vartheta D(X)\to\T W_D(X)
\end{equation*}
for each linear order~$X$. Assume that $X$ is a well order and hence a well partial order. Then $\T W_D(X)$ is also a well partial order, by~\cite[Proposition~2.7]{freund-kruskal-gap} (which is proved by the minimal bad sequence method of Nash-Williams~\cite{nash-williams63}). Given an infinite sequence $s_0,s_1,\dots$ in $\vartheta D(X)$, we get a sequence $\nu^+_X(s_0),\nu^+_X(s_1),\dots$ in~$\T W_D(X)$. Since the latter is a well partial order, there are $i<j$ such that we have $\nu^+_X(s_i)\leq_{\T W_D(X)}\nu^+_X(s_j)$. We can infer $s_i\leq_{\vartheta D(X)}s_j$, as~$\nu^+_X$ is a quasi embedding and hence order reflecting. This shows that~$\vartheta D(X)$ is a well order.
\end{proof}

As an immediate consequence of our general approach, we re-obtain the following known result about the orders $T_n[0]$ from~\cite[Section~2.3.3]{MRW-linear}.

\begin{corollary}\label{cor:collapsing-well-founded}
The order~$T_n[0]$ is well founded for each $n\in\mathbb N$.
\end{corollary}
\begin{proof}
In the previous section we have studied $\lo$-dilators~$\te^0_n$ with $T_n[0]\cong\te^0_n(1)$. It suffices to show that these are $\mathsf{WO}$-dilators. We argue by induction on~$n$. In the base we need only observe $\te^0_n(X)\cong X$. The step is covered by the previous corollary, as $\te^0_{n+1}$ is the Bachmann-Howard derivative of~$\te^0_n$, by Theorem~\ref{thm:T-n+1-deriv}.
\end{proof}

Let $f:X\to Y$ be a quasi embedding between partial orders. If $X$ is linear and $f$ is surjective, then we call $f$ a linearization (of~$Y$ by~$X$). This coincides with the usual notion if we identify~$X$ with its image under~$f$. Linearizations are particularly important, as they are related to maximal order types and independence results (see e.\,g.~\cite{deJongh-Parikh,simpson85,aschenbrenner-pong,knight-lange}). We introduce the corresponding functorial notion:

\begin{definition}\label{def:linearization-dil}
Consider an $\lo$-dilator~$D$ and a $\po$-dilator~$W$. A quasi embedding $\nu:D\To W\!\restriction\!\lo$ is called a linearization (of~$W$ by~$D$) if $\nu_X:D(X)\to W(X)$ is surjective for each linear order~$X$.
\end{definition}

Our next goal is to identify a condition which ensures that the quasi embedding $\nu^+:\vartheta D\To\mathcal T W\!\restriction\!\lo$ from Theorem~\ref{thm:quasi-embeddings-derivatives} is a linearization. As the following example shows, the assumption that $\nu:D\To W\!\restriction\!\lo$ is a linearization does not suffice.

\begin{example}
We define a transformation $X\mapsto W(X)$ of partial orders by
\begin{gather*}
W(X)=\{(x,x')\in X^2\,|\,x'\not\leq_Xx\},\\
(x,x')\leq_{W(X)}(y,y')\quad\Leftrightarrow\quad x\leq_X y\text{ and }x'\leq_X y'.
\end{gather*}
If $f:X\to Y$ is a quasi embedding, then $x'\not\leq_Xx$ implies $f(x')\not\leq_Yf(x)$. We thus get a function $W(f):W(X)\to W(Y)$ by setting $W(f)(x,x')=(f(x),f(x'))$. Let us also define functions $\supp^W_X:W(X)\to[X]^{<\omega}$ by putting $\supp^W_X(x,x')=\{x,x'\}$. It is straightforward to verify that this data constitutes a normal $\po$-dilator. Here it is crucial that the support condition
\begin{equation*}
\supp^W_Y(\sigma)\subseteq\rng(f)\quad\To\quad\sigma\in\rng(W(f))
\end{equation*}
is only required when $f:X\to Y$ is an embedding, rather than just a quasi embedding (see \cite[Definition~2.1]{freund-kruskal-gap}). In the present case, the support condition does in fact imply that $f$ is order preserving. For each linear order~$X$, we put
\begin{gather*}
D(X)=\{(x,x')\in X^2\,|\,x<_Xx'\},\\
(x,x')\leq_{D(X)}(y,y')\quad\Leftrightarrow\quad x<_X y\text{ or }(x=y\text{ and }x'\leq_X y').
\end{gather*}
Given that~$X$ is linear, the underlying sets of $D(X)$ and $W(X)$ are equal. We may thus declare that $D(f):D(X)\to D(Y)$ and $\supp^D_X:D(X)\to[X]^{<\omega}$ coincide with $W(f)$ and $\supp^W_X$, respectively, for any embedding $f:X\to Y$ of linear orders. It is easy to see that this yields an $\lo$-dilator. The identity maps $\nu_X:D(X)\to W(X)$ constitute a linearization $\nu:D\To W\!\restriction\!\lo$. Theorem~\ref{thm:quasi-embeddings-derivatives} gives rise to a quasi embedding $\nu^+:\vartheta D\To\T W\!\restriction\!\lo$, which we shall now describe concretely. For each partial order~$X$, we consider a collection $\mathcal B(X)$ of structured binary trees with leaf labels from~$X$, which is recursively generated as follows:
\begin{itemize}[label={--}]
\item For each $x\in X$ we have an element~$\overline x\in\mathcal B(X)$ (a single node with label~$x$).
\item Given $s,s'\in\mathcal B(X)$ we add an element $\circ(s,s')\in\mathcal B(X)$ (the tree in which the root has immediate subtrees~$s$ and~$s'$).
\end{itemize}
To extend $\mathcal B$ into a functor (currently on sets but later on orders), we stipulate
\begin{equation*}
\mathcal B(f)(\overline x)=\overline{f(x)}\quad\text{and}\quad\mathcal B(f)(\circ(s,s'))=\circ(\mathcal B(f)(s),\mathcal B(f)(s')).
\end{equation*}
We also define support functions $\supp^{\mathcal B}_X:\mathcal B(X)\to[X]^{<\omega}$ by setting
\begin{equation*}
\supp^{\mathcal B}_X(\overline x)=\{x\}\quad\text{and}\quad\supp^{\mathcal B}_X(\circ(s,s'))=\supp^{\mathcal B}_X(s)\cup\supp^{\mathcal B}_X(s').
\end{equation*}
To describe the Kruskal derivative $\T W$ of~$W$, we first define a partial order~$\leq^W_{\mathcal B(X)}$ on $\mathcal B(X)$. The latter coincides with the usual embeddability relation for labelled trees, which can be given by the following recursive clauses:
\begin{enumerate}[label=(\roman*)]
\item If we have $x\leq_X y$, then we have $\overline x\leq^W_{\mathcal B(X)}\overline y$.
\item We have $s\leq^W_{\mathcal B(X)}\circ(s,s')$ and $s'\leq^W_{\mathcal B(X)}\circ(s,s')$.
\item If we have $s\leq^W_{\mathcal B(X)}t$ and $s'\leq^W_{\mathcal B(X)}t'$, then we have $\circ(s,s')\leq^W_{\mathcal B(X)}\circ(t,t')$.
\end{enumerate}
One can check that this turns $\mathcal B$ into a normal $\po$-dilator. Indeed, $\mathcal B$ is the Kruskal derivative of $X\mapsto X^2$, where the order on~$X^2$ extends the one on $W(X)$ in the obvious way. In view of $W(X)\subsetneq X^2$, we now define $\T W(X)\subseteq\mathcal B(X)$ as the smallest suborder that contains all elements~$\overline x$ and contains $\circ(s,s')$ for any $s,s'\in\T W(X)$ with $s'\not\leq^W_{\mathcal B(X)} s$. We can turn $\T W$ into a normal $\po$-dilator by restricting the constructions from above. Let us define
\begin{equation*}
\iota^W_X:X\to\T W(X)\quad\text{and}\quad\kappa^W_X:W\circ\T W(X)\to\T W(X)
\end{equation*}
by setting $\iota^W_X(x)=\overline x$ and $\kappa^W_X(s,s')=\circ(s,s')$. As one readily verifies, these functions witness that $\T W$ is the Kruskal derivative of~$W$ (so that the notation $\T W$ is indeed justified). To describe the Bachmann-Howard derivative $\vartheta D$ of~$D$, we first refine $\leq^W_{\mathcal B(X)}$ into a linear order $\leq_{\mathcal B(X)}^D$ on $\mathcal B(X)$. The latter is characterized by the following clauses together with (i-iii) above (with $\leq^D_{\mathcal B(X)}$ at the place of $\leq^W_{\mathcal B(X)}$):
\begin{enumerate}[label=(\roman*)]\setcounter{enumi}{3}
\item We have $\overline x<_{\mathcal B(X)}^D\circ(t,t')$ for any terms of the given forms.
\item If we have $s<_{\mathcal B(X)}^Dt$ and $s'<_{\mathcal B(X)}^D\circ(t,t')$, we have $\circ(s,s')<^D_{\mathcal B(X)}\circ(t,t')$.
\end{enumerate}
Analogous to the above, let $\vartheta D(X)\subseteq\mathcal B(X)$ be the smallest (linear) suborder that contains all elements~$\overline x$ and contains $\circ(s,s')$ for any $s,s'\in\vartheta D(X)$ with $s<^D_{\mathcal B(X)} s'$. To turn $\vartheta D$ into an $\lo$-dilator, it suffices to restrict the functions $\mathcal B(f)$ and $\supp^{\mathcal B}_X$ from above. We define
\begin{equation*}
\iota^D_X:X\to\vartheta D(X)\quad\text{and}\quad\vartheta^D_X:D\circ\vartheta D(X)\to\vartheta D(X)
\end{equation*}
as the restrictions of $\iota^W_X$ and $\kappa^W_X$, i.\,e.~by $\iota^D_X(x)=\overline x$ and $\vartheta^D_X(s,s')=\circ(s,s')$. These functions witness that $\vartheta D$ is the Bachmann-Howard derivative of~$D$. Since $\leq^D_{\mathcal B(X)}$ refines $\leq^W_{\mathcal B(X)}$, we have $\vartheta D(X)\subseteq\T W(X)$. One can verify that the inclusion maps
\begin{equation*}
\nu^+_X:\vartheta D(X)\hookrightarrow\T W(X)
\end{equation*}
satisfy the conditions from Theorem~\ref{thm:quasi-embeddings-derivatives}, which characterize $\nu^+$ uniquely. We now consider $X=\{0,1,2\}$ with the usual linear order. The element~$\circ(\circ(\overline 0,\overline 1),\overline 2)\in\mathcal B(X)$ is contained in $\T W(X)$ but not in $\vartheta D(X)$, so that $\nu^+_X$ is not surjective. This means that $\nu^+:\vartheta D\To\T W\!\restriction\!\lo$ is no linearization, even though $\nu:D\To W\!\restriction\!\lo$ is one.
\end{example}

The previous example suggests the following notion:

\begin{definition}\label{def:dil-flat}
A $\po$-dilator~$W$ is called flat if the support condition
\begin{equation*}
\supp^W_Y(\sigma)\subseteq\rng(f)\quad\To\quad\sigma\in\rng(W(f))
\end{equation*}
holds for any quasi embedding $f:X\to Y$ between partial orders (recall that \cite[Definition~2.1]{freund-kruskal-gap} does only require this condition for embeddings).
\end{definition}

As expected, we get the following (cf.~Theorem~\ref{thm:quasi-embeddings-derivatives}):

\begin{theorem}\label{thm:linearization-preserved}
Consider a linearization $\nu:D\To W\!\restriction\!\lo$ of a normal $\po$-dilator~$W$ by an $\lo$-dilator~$D$. If $W$ is flat, then $\nu^+:\vartheta D\To\T W\!\restriction\!\lo$ is again a linearization.
\end{theorem}
\begin{proof}
Given a linear order~$X$, we need to show that $\nu^+_X:\vartheta D(X)\to\T W(X)$ is surjective. According to~\cite[Definition~4.1]{freund-kruskal-gap}, the order~$T W(X)$ comes with functions $\iota^W_X:X\to\T W(X)$ and $\kappa^W_X:W\circ\T W(X)\to\T W(X)$ that turn it into an initial Kruskal fixed point. By~\cite[Theorem~3.5]{freund-kruskal-gap} we get a function $h:\T W(X)\to\mathbb N$ with
\begin{equation*}
r\in\supp^W_{\T W(X)}(\sigma)\quad\To\quad h(r)<h(\kappa^W_X(\sigma)).
\end{equation*}
For $s\in\T W(X)$ we shall now show $s\in\rng(\nu^+_X)$ by induction on~$h(s)$. Also by~\cite[Theorem~3.5]{freund-kruskal-gap}, it suffices to consider the following two cases: First, let us look at an element~$s=\iota^W_X(x)$. Here we obtain $s=\nu^+_X(\iota^D_X(x))\in\rng(\nu^+_X)$ by the diagram in Theorem~\ref{thm:quasi-embeddings-derivatives}. Secondly, consider $s=\kappa^W_X(\sigma)$. The induction hypothesis yields
\begin{equation*}
\supp^W_{\T W(X)}(\sigma)\subseteq\rng(\nu^+_X).
\end{equation*}
Let us recall that $\nu^+_X$ is a quasi embedding but no embedding (unless~$\T W(X)$ is linear). Hence the support condition from \cite[Definition~2.1]{freund-kruskal-gap} does not apply. But since~$W$ is flat we get $\sigma\in\rng(W(\nu^+_X))$ anyway. We also know that $\nu_{\vartheta D(X)}$ is surjective, by the assumption that $\nu$ is a linearization. Thus we may write
\begin{equation*}
\sigma=W(\nu^+_X)\circ\nu_{\vartheta D(X)}(\sigma_0)\quad\text{for some}\quad\sigma_0\in D\circ\vartheta D(X).
\end{equation*}
Due to the diagram in Theorem~\ref{thm:quasi-embeddings-derivatives}, we then obtain
\begin{equation*}
s=\kappa^W_X(\sigma)=\nu^+_X\circ\vartheta^D_X(\sigma_0)\in\rng(\nu^+_X),
\end{equation*}
as needed to complete the inductive proof.
\end{proof}

In the following section we will construct sequences with gap condition via iterated Kruskal derivatives. The following result will ensure that all $\po$-dilators in this construction are flat.

\begin{proposition}\label{prop:flat-derivative}
We consider a normal $\po$-dilator~$W$. If $W$ is flat, then so is its Kruskal derivative~$\T W$.
\end{proposition}
\begin{proof}
Given a quasi embedding~$f:X\to Y$, we need to prove the support condition
\begin{equation*}
\supp^{\T W}_Y(s)\subseteq\rng(f)\quad\To\quad s\in\rng(\T W(f))
\end{equation*}
for arbitrary~$s\in\T W(Y)$. Similarly to the previous proof, this can be achieved by induction on~$h(s)$, where $h:\T W(Y)\to\mathbb N$ with
\begin{equation*}
r\in\supp^W_{\T W(Y)}(\sigma)\quad\To\quad h(r)<h(\kappa^W_Y(\sigma))
\end{equation*}
is provided by~\cite[Theorem~3.5]{freund-kruskal-gap}. An element $s=\iota^W_Y(y)$ has support $\{y\}$, by the proof of~\cite[Theorem~4.2]{freund-kruskal-gap}. Given that the premise of the support condition holds, we thus get $y=f(x)$ for some~$x\in X$. The diagram in~\cite[Definition~4.1]{freund-kruskal-gap} yields
\begin{equation*}
s=\iota_Y\circ f(x)=\T W(f)\circ\iota_X(x)\in\rng(\T W(f)),
\end{equation*}
as required. It remains to consider an element $s=\kappa^W_Y(\sigma)$. Here we have
\begin{equation*}
\supp^{\T W}_Y(s)=\bigcup\{\supp^{\T W}_X(r)\,|\,r\in\supp^W_{\T W(X)}(\sigma)\},
\end{equation*}
again by the proof of~\cite[Theorem~4.2]{freund-kruskal-gap}. Assuming the premise of the support condition for~$s$, we can thus invoke the induction hypothesis to get
\begin{equation*}
\supp^W_{\T W(X)}(\sigma)\subseteq\rng(\T W(f)).
\end{equation*}
Using the assumption that~$W$ is flat, we obtain
\begin{equation*}
\sigma=W\circ\T W(f)(\sigma_0)\quad\text{for some}\quad\sigma_0\in W\circ\T W(X).
\end{equation*}
Now the diagram in~\cite[Definition~4.1]{freund-kruskal-gap} yields
\begin{equation*}
s=\kappa^W_Y(\sigma)=\T W(f)\circ\kappa^W_X(\sigma_0)\in\rng(\T W(f)),
\end{equation*}
as needed for the inductive proof that $\T W$ is flat.
\end{proof}

\section{Application, part~2: linearizing the gap condition for sequences}\label{sect:application-2}

As shown in Section~\ref{sect:collapsing-functions}, one can reconstruct certain collapsing functions by taking iterated Bachmann-Howard derivatives (of dilators on linear orders). In the present section we show that sequences with Friedman's gap condition arise from a completely parallel construction in terms of Kruskal derivatives (of dilators on partial orders). This confirms that there is an extremely tight connection between collapsing functions and the gap condition. Also in this section, we give a second iterative construction that yields the same gap condition but different collapsing functions. This explains a phenomenon from~\cite{MRW-linear}, as discussed in the introduction of the present paper.

\begin{definition}\label{def:se-order}
For each partial order~$X$, let $\se(X)$ and $S:\se(X)\to\mathbb N\cup\{-1\}$ be generated by clauses~(i) and~(ii) from Definition~\ref{def:term-system-T} (with $\se$ at the place of~$\te$, so that $\se(X)=\te(X)$ when $X$ is linear). Also, let $k_i:\se(X)\to\se(X)$ for $i\in\mathbb N\cup\{-1\}$ be given as in the cited definition. To define a binary relation $\leq_{\se(X)}$ on~$\se(X)$, we declare that $s\leq_{\se(X)}t$ holds if, and only if, one of the following clauses is satisfied:
\begin{enumerate}[label=(\roman*')]
\item We have $s=\overline x$ and $t=\overline y$ with $x\leq_X y$.
\item We have $t=\vartheta_jt'$ with $s\leq_{\se(X)}k_j(t')$ (where $s$ can be of the form $\overline x$ or $\vartheta_is'$).
\item We have $s=\vartheta_is'$ and $t=\vartheta_it'$ (for the same~$i$) with $s'\leq_{\se(X)}t'$.
\end{enumerate}
Furthermore, let~$\se_n(X)\subseteq\se(X)$ consist of the terms with indices $i<n$ only, and put~$\se^0_n(X)=\{s\in\se_n(X)\,|\,S(s)\leq 0\}$ (cf.~Definition~\ref{def:T_n}). We will always consider these subsets with (the restrictions of) the relation~$\leq_{\se(X)}$ (sometimes written as~$\leq$).
\end{definition}

Let $h:\se(X)\to\mathbb N$ be given as in the paragraph after Definition~\ref{def:term-system-T}. Once again we have $h(k_i(s))\leq h(s)$, so that $s\leq_{\se(X)}t$ can be decided by recursion on~$h(s)+h(t)$. Let us begin with a preparatory result:

\begin{lemma}\label{lem:k_i-monotone}
If we have $s\leq_{\se(X)}t$, then we get $k_i(s)\leq_{\se(X)}k_i(t)$ for any $i\in\mathbb N$.
\end{lemma}
\begin{proof}
We argue by induction on~$h(s)+h(t)$. First, assume that $s\leq\vartheta_jt'=t$ holds because of~$s\leq k_j(t')$. We then get $k_i(s)\leq k_i(k_j(t'))$ by induction hypothesis. For $j\leq i$ we have $k_i(k_j(t'))=k_j(t')$, so that we obtain $k_i(s)\leq k_j(t')$ and then
\begin{equation*}
k_i(s)\leq\vartheta_jt'=k_i(\vartheta_jt')=k_i(t).
\end{equation*}
For $j>i$ we get $k_i(k_j(t'))=k_i(t')$ and thus $k_i(s)\leq k_i(t')=k_i(t)$. Secondly, assume that $s=\vartheta_js'\leq\vartheta_jt'=t$ holds because of $s'\leq t'$. For $i<j$ we have $k_i(s)=k_i(s')$ and $k_i(t)=k_i(t')$, so that it suffices to invoke the induction hypothesis. For $i\geq j$ we have $k_i(s)=s$ and $k_i(t)=t$, which means that the claim is trivial.
\end{proof}

The following is known for~$X=1=\{0\}$ (see e.\,g.~\cite{MRW-linear}), but the author has not found a reference for the (straightforward) generalization.

\begin{lemma}\label{lem:se-partial-order}
The relation $\leq_{\se(X)}$ on $\se(X)$ is a partial order.
\end{lemma}
\begin{proof}
Reflexivity is immediate by induction over terms. To establish antisymmetry one first shows that $s\leq t$ entails $h(s)\leq h(t)$, by induction over~$h(s)+h(t)$. The point is that $t\leq s$ becomes impossible if $s\leq t=\vartheta_jt'$ holds by clause~(ii') above, since we then have
\begin{equation*}
h(s)\leq h(k_j(t'))\leq h(t')<h(t).
\end{equation*}
Antisymmetry between $s$ and $t$ is now straightforward by induction on~$h(s)+h(t)$. To show that $r\leq s$ and $s\leq t$ imply $r\leq t$, one uses induction on $h(r)+h(s)+h(t)$. In the most interesting case we have $r\leq s=\vartheta_is'$ because of $r\leq k_i(s')$, while $s\leq t=\vartheta_i t'$ is due to $s'\leq t'$. The previous lemma ensures $k_i(s')\leq k_i(t')$. Using the induction hypothesis, we get $r\leq k_i(t')$ and then $r\leq\vartheta_it'=t$.
\end{proof}

For $X=1=\{0\}$ we get $\overline 0\leq_{\se(1)}s$ by a straightforward induction over~$s\in\se(X)$. This means that our order~$(\se_n(1),\leq_{\se(1)})$ coincides with the order~$(T_n,\trianglelefteq)$ defined in~\cite[Section~2.3.3]{MRW-linear}. According to~\cite[Lemma~10]{MRW-linear}, we thus have the following (see the cited reference for a precise definition of the gap condition):

\begin{corollary}\label{cor:gap-condition-coincides}
If we identify $\vartheta_{i_1}\ldots\vartheta_{i_n}0\in\se_n(1)$ with the sequence $\langle i_1,\dots,i_n\rangle$, then $\leq_{\se(1)}$ coincides with the strong gap embeddability relation due to H.~Friedman.
\end{corollary}

Analogous to Section~\ref{sect:collapsing-functions}, we will show that $\se^0_{n+1}$ is the Kruskal derivative of~$\se^0_n$. Let us first turn $\se_n$ and $\se^0_n$ into $\po$-dilators.

\begin{definition}\label{def:se-dilators}
For each quasi embedding $f:X\to Y$ between partial orders, let $\se(f):\se(X)\to\se(Y)$ be defined by the clauses from Definition~\ref{def:T-dilator} (with $\se$ at the place of~$\te$). Let $\underline k:\se(X)\to X$ be given as in the paragraph before Proposition~\ref{prop:T-deriv-prep}. For a partial order~$X$, we now define $\supp^\se_X:\se(X)\to[X]^{<\omega}$ by $\supp^\se_X(s)=\{\underline k(s)\}$. We will also write $\supp^\se_X$ for the restrictions of this function to $\se_n(X)$ and to $\se^0_n(X)$. Furthermore, we write $\se_n(f):\se_n(X)\to\se_n(Y)$ and $\se^0_n(f):\se^0_n(X)\to\se^0_n(Y)$ for the restrictions of~$\se(f)$ with the indicated (co-)domains.
\end{definition}

As expected, we have the following:

\begin{proposition}
Definition~\ref{def:se-dilators} extends $\se_n$ and $\se^0_n$ into normal~$\po$-dilators.
\end{proposition}
\begin{proof}
Most verifications are completely parallel to the proof of Proposition~\ref{prop:te-dilators}, but two additional observations are needed: First, the map $f\mapsto\se(f)$ on morphisms preserves not only embeddings but also quasi embeddings, as required by clause~(i) of~\cite[Definition~2.1]{freund-kruskal-gap}. Secondly, normality amounts to the implication
\begin{equation*}
s\leq_{\se(X)}t\quad\To\quad\underline k(s)\leq_X\underline k(t),
\end{equation*}
which holds by Lemma~\ref{lem:k_i-monotone} (since we have $\underline k(r)=x$ for $k_{-1}(r)=\overline x$).
\end{proof}

Let us now describe the extension into a Kruskal derivative:

\begin{definition}
For each number~$n\in\mathbb N$ and any partial order~$X$, let
\begin{equation*}
\sigma^n_X:\se_n\circ\se^0_{n+1}(X)\to\se_{n+1}(X)
\end{equation*}
be given by the clauses from Definition~\ref{def:iota-theta}. Now define functions
\begin{equation*}
\kappa^n_X:\se^0_n\circ\se^0_{n+1}(X)\to\se^0_{n+1}(X)\quad\text{and}\quad\iota^n_X:X\to\se^0_{n+1}(X)
\end{equation*}
by setting $\kappa^n_X(s)=\vartheta_0\sigma^n_X(s)$ and $\iota^n_X(x)=\overline x$.
\end{definition}

Arguing as in the proof of Proposition~\ref{prop:T-deriv-prep}, one can show the following:

\begin{proposition}
The function $\sigma^n_X:\se_n\circ\se^0_{n+1}(X)\to\se_{n+1}(X)$ is an order isomorphism, for each number~$n\in\mathbb N$ and any partial order~$X$.
\end{proposition}
\begin{proof}
To show that $\sigma^n_X$ is surjective and order preserving, one argues as in the proof of Proposition~\ref{prop:T-deriv-prep}. It remains to show that $\sigma^n_X$ reflects the order, i.\,e., that
\begin{equation*}
\sigma^n_X(s)\leq\sigma^n_X(t)\quad\To\quad s\leq t
\end{equation*}
holds for $s,t\in\se_n\circ\se^0_{n+1}(X)$. We argue by induction on $h(s)+h(t)$ and discuss the two most interesting cases: First consider terms of the form $s=\overline r$ and $t=\vartheta_jt'$. In view of $r\in\te^0_{n+1}(X)$ we have $r=\overline x$ or $\vartheta_0r'$, which means that
\begin{equation*}
\sigma^n_X(s)=r\leq\vartheta_{j+1}\sigma^n_X(t')=\sigma^n_X(t)
\end{equation*}
can only be due to $\sigma^n_X(s)\leq k_{j+1}(\sigma^n_X(t'))=\sigma^n_X(k_j(t'))$, where the equality comes from the proof of Proposition~\ref{prop:T-deriv-prep}. By induction hypothesis we get $s\leq k_j(t')$ and then $s\leq\vartheta_jt'=t$, as desired. Let us also consider $s=\vartheta_is'$ and $t=\overline r$. For terms of these forms, the inequality $s\leq t$ is always false. To show that $\sigma^n_X(s)\leq\sigma^n_X(t)$ is false as well, we observe
\begin{equation*}
S(\sigma^n_X(s))=S(\vartheta_{i+1}\sigma^n_X(s'))=i+1>0\geq S(r)=S(\sigma^n_X(t)),
\end{equation*}
where $S(r)\leq 0$ is due to~$r\in\te^0_n(X)$. On the other hand, a straightforward induction on $h(s_0)+h(t_0)$ shows that $s_0\leq t_0$ implies $S(s_0)\leq S(s_1)$.
\end{proof}

In the following, $\iota^n$ and $\kappa^n$ denote the families of functions $\iota^n_X:X\to\se^0_{n+1}(X)$ and $\kappa^n_X:\se^0_n\circ\se^0_{n+1}(X)\to\se^0_{n+1}(X)$, respectively, indexed by the partial order~$X$. We recall from~\cite[Section~4]{freund-kruskal-gap} that Kruskal derivatives are essentially unique. Hence the following means that the recursive construction from the introduction is uniquely realized by the normal $\po$-dilators~$\se^0_n$ that we have defined in the present section.

\begin{theorem}\label{thm:S-n+1-deriv}
The Kruskal derivative of~$\se^0_n$ is given by $(\se^0_{n+1},\iota^n,\kappa^n)$.
\end{theorem}
\begin{proof}
Part~(i) of~\cite[Definition~4.1]{freund-kruskal-gap} requires that $(\se^0_{n+1}(X),\iota^n_X,\kappa^n_X)$ is an initial Kruskal fixed point of $\se^0_n$ over~$X$, for each partial order~$X$. The most interesting conditions from the definition of Kruskal fixed point (see~\cite[Definition~3.1]{freund-kruskal-gap}) demand
\begin{align*}
\overline x\leq\kappa^n_X(t)\quad&\Leftrightarrow\quad\overline x\leq\underline k(t),\\
\kappa^n_X(s)\leq\kappa^n_X(t)\quad&\Leftrightarrow\quad s\leq t\text{ or }\kappa^n_X(s)\leq\underline k(t)
\end{align*}
for $x\in X$ and $s,t\in\se^0_n\circ\se^0_{n+1}(X)$. In view of $\kappa^n_X(t)=\vartheta_0\sigma^n_X(t)$, the left side of the first equivalence is indeed equivalent to
\begin{equation*}
\overline x\leq k_0(\sigma^n_X(t))=\sigma^n_X(k_{-1}(t))=\underline k(t),
\end{equation*}
as in the proof of Proposition~\ref{prop:T-deriv-prep}. Similarly, the left side of the second equivalence is equivalent to the disjunction of $\sigma^n_X(s)\leq\sigma^n_X(t)$ and $\kappa^n_X(s)\leq\underline k(t)$. This is equivalent to the right side, as $\sigma^n_X$ is an order embedding. To show that our fixed point is initial we use the criterion from~\cite[Theorem~3.5]{freund-kruskal-gap}, which demands that we have $\se^0_{n+1}(X)=\rng(\iota^n_X)\cup\rng(\kappa^n_X)$ and $h(\underline k(s))<h(\kappa^n_X(s))$ for some $h:\se^0_{n+1}(X)\to\mathbb N$. These requirements can be verified as in the proof of Theorem~\ref{thm:T-n+1-deriv}. It remains to show that part~(ii) of~\cite[Definition~4.1]{freund-kruskal-gap} is satisfied, i.\,e., that the functions $\iota^n_X$ and $\kappa^n_X$ are natural in~$X$. This is readily reduced to the naturality of~$\sigma^n_X$, which can be established by induction over terms, as in the proof of Theorem~\ref{thm:T-n+1-deriv}.
\end{proof}

We can now use our general approach to re-derive two results that were previously shown by explicit computations. In view of Corollaries~\ref{cor:coincide-mrw-linear} and~\ref{cor:gap-condition-coincides}, we focus on the case of~$X=1$. The following is a special case of H.~Friedman's result on tree embeddings with the gap condition (see~\cite{simpson85}). For the case of sequences, the result was analyzed by Sch\"utte and Simpson~\cite{schuette-simpson} (see also~\cite[Section~2.2]{MRW-linear}).

\begin{corollary}
The order $\se^0_n(1)$ (sequences with gap condition, cf.~Corollary~\ref{cor:gap-condition-coincides}) is a well partial order for each~$n\in\mathbb N$.
\end{corollary}
\begin{proof}
We argue by induction on $n$ to show that $X\mapsto\se^0_n(X)$ preserves well partial orders. The base case is immediate in view of~$\se^0_n(X)\cong X$. The induction step is covered by~\cite[Corollary~4.5]{freund-kruskal-gap}, given that $\se^0_{n+1}$ is the Kruskal derivative of $\se^0_n$.
\end{proof}

As in the previous section, a surjective quasi embedding $f:X\to Y$ between a linear order~$X$ and a partial order~$Y$ is called a linearization. The following result was first established by a concrete verification in~\cite[Lemma~11]{MRW-linear}, which provides additional information: it shows that one can take $\nu^n_1$ to be the identity on the underlying set~$\te^0_n(1)=\se^0_n(1)$. This information could also be tracked through our general constructions, but we will not do so: the forte of our approach is precisely that fewer concrete computations are necessary.

\begin{corollary}
For each $n\in\mathbb N$ we have a linearization $\nu^n_1:\te^0_n(1)\to\se^0_n(1)$ of the partial order~$\se^0_n(1)$ (sequences with gap condition) by the linear order $\te^0_n(1)$ (collapsing functions, cf.~Section~\ref{sect:collapsing-functions}).
\end{corollary}
\begin{proof}
As preparation, we use induction on~$n$ to show that $\se^0_n$ is flat in the sense of Definition~\ref{def:dil-flat}. The induction step is covered by Proposition~\ref{prop:flat-derivative}. For the base case we recall $\se^0_0(Y)=\{\overline y\,|\,y\in Y\}$. Given~$f:X\to Y$ with 
\begin{equation*}
\supp^\se_Y(\overline y)=\{\underline k(\overline y)\}=\{y\}\subseteq\rng(f),
\end{equation*}
we write $y=f(x)$ to get $\overline y=\overline{f(x)}=\se^0_0(f)(\overline x)\in\rng(\se^0_0(f))$. The point is that this works for any quasi embedding~$f$, not just for embeddings. By recursion on~$n$, we now construct linearizations $\nu^n:\te^0_n\To\se^0_n\!\restriction\!\lo$ in the sense of Definition~\ref{def:linearization-dil}. For $n=0$ we can define $\nu^0_X$ as the identity map on $\te^0_0(X)=\{\overline x\,|\,x\in X\}=\se^0_0(X)$. In the step we define $\nu^{n+1}$ as the quasi embedding $(\nu^n)^+$ from Theorem~\ref{thm:quasi-embeddings-derivatives}. This is justified because $\te^0_{n+1}$ and $\se^0_{n+1}$ are the Bachmann-Howard and Kruskal derivative of $\te^0_n$ and $\se^0_n$, respectively, by Theorems~\ref{thm:T-n+1-deriv} and~\ref{thm:S-n+1-deriv}. From Theorem~\ref{thm:linearization-preserved} we learn that $\nu^{n+1}$ is a linearization, given that the same holds for $\nu^n$ and that $\se^0_n$ is flat.
\end{proof}

The linearization of $\se^0_n(1)$ by $\te^0_n(1)$ does not realize the maximal order type, as shown in~\cite{MRW-linear} (see also the introduction of the present paper). We will now present a different construction of $\se^0_n(1)$ in terms of iterated Kruskal fixed points. Interestingly, the parallel construction in terms of Bachmann-Howard fixed points yields the order $OT_n[0]$ from~\cite[Section~5]{MRW-linear}, which differs from $\te^0_n(1)$ and does realize the maximal order type of~$\se^0_n(1)$. The idea is to consider the following transformations~$W_n$, which are readily shown to be normal $\po$-dilators.

\begin{definition}
For each~$n\in\mathbb N$ and any partial order~$X$ we put
\begin{equation*}
W_n(X)=1+\se^0_n(1)\times X=\{0\}\cup\{(s,x)\,|\,s\in\se^0_n(1)\text{ and }x\in X\}.
\end{equation*}
To define a partial order on~$W_n(X)$, we declare that $0$ is incomparable to all other elements, and that we have
\begin{equation*}
(s,x)\leq_{W_n(X)}(s',x')\quad\Leftrightarrow\quad s\leq_{\se(1)}s'\text{ and }x\leq_X x'.
\end{equation*}
Given a quasi embedding $f:X\to Y$, we define $W_n(f):W_n(X)\to W_n(Y)$ by $W_n(f)(0)=0$ and $W_n(f)(s,x)=(s,f(x))$. Also, let $\supp_X:W_n(X)\to[X]^{<\omega}$ be given by $\supp_X(0)=\emptyset$ and $\supp_X(s,x)=\{x\}$.
\end{definition}

We will show that $\se^0_{n+1}(1)$ is the initial Kruskal fixed point of $W_n$ over $0=\emptyset$. This gives rise to a recursive construction, since $W_n$ does only depend on~$\se^0_n(1)$, and since initial objects are unique up to isomorphism.

\begin{definition}
To construct a function $\pi_n:\se_n(1)\times\se^0_{n+1}(1)\to\se_{n+1}(1)$, we define $\pi_n(s,t)$ by recursion over the term~$s$, setting
\begin{equation*}
\pi_n(\overline 0,t)=t\quad\text{and}\quad\pi_n(\vartheta_is,t)=\vartheta_{i+1}\pi_n(s,t).
\end{equation*}
We then define $\kappa_n:W_n(\se^0_{n+1}(1))\to\se^0_{n+1}(1)$ by $\kappa_n(0)=\overline 0$ and $\kappa_n(s,t)=\vartheta_0\pi_n(s,t)$. Let us also agree to write $\iota_n:0\to\se^0_{n+1}(1)$ for the empty function.
\end{definition}

To see that the definition of $\pi_n$ is justified, one should observe that we have
\begin{equation*}
S(\pi_n(s,t))=\begin{cases}
S(t)\leq 0 & \text{if $s=0$},\\
S(s)+1 & \text{otherwise}.
\end{cases}
\end{equation*}
This also shows that $s\in\se^0_n(1)$ entails $S(\pi_n(s,t))\leq 1$, as needed to justify the definition of~$\kappa_n$. As promised, we have the following:

\begin{theorem}\label{thm:se_n-fixed-point-2}
The tuple $(\se^0_{n+1}(1),\iota_n,\kappa_n)$ is the initial Kruskal fixed point of $W_n$ over~$0$, for each number~$n\in\mathbb N$.
\end{theorem}
\begin{proof}
Let us abbreviate $Z=\se^0_{n+1}(1)$. To show that we have a Kruskal fixed point, we need to prove that the equivalence
\begin{equation*}
\kappa_n(\sigma)\leq_Z\kappa_n(\tau)\quad\Leftrightarrow\quad\sigma\leq_{W_n(Z)}\tau\text{ or }\kappa_n(\sigma)\leq_Z t\text{ for some }t\in\supp_Z(\tau)
\end{equation*}
is satisfied for all $\sigma,\tau\in W_n(Z)$ (cf.~\cite[Definition~3.1]{freund-kruskal-gap}). For $\sigma=0$, both sides of the equivalence hold, since we have $\overline 0\leq_Z t$ for any $t\in Z$ (see the paragraph before Corollary~\ref{cor:gap-condition-coincides}). For $\sigma\neq 0$ and $\tau=0$, both sides of the equivalence fail. Now consider $\sigma=(s,t)$ and $\tau=(s',t')$. Then our equivalence amounts to
\begin{equation*}
\vartheta_0\pi_n(s,t)\leq_Z\vartheta_0\pi_n(s',t')\quad\Leftrightarrow\quad(s\leq_{\se(1)}s'\text{ and }t\leq_Z t')\text{ or }\vartheta_0\pi_n(s,t)\leq_Z t'.
\end{equation*}
Given that $t'\in\se^0_{n+1}(1)$ entails $k_{i+1}(t')=t'$, an auxiliary induction over $s'$ yields
\begin{equation*}
k_{i+1}(\pi_n(s',t'))=\pi_n(k_i(s'),t').
\end{equation*}
For $i=-1$ we get $k_{-1}(s')=\overline 0$ and hence $k_0(\pi_n(s',t'))=\pi_n(\overline 0,t')=t'$. In view of Definition~\ref{def:se-order}, this entails that the last equivalence reduces to
\begin{equation*}
\pi_n(s,t)\leq_{\se(1)}\pi_n(s',t')\quad\Leftrightarrow\quad s\leq_{\se(1)}s'\text{ and }t\leq_Z t'.
\end{equation*}
The latter can be shown by induction over $h(s)+h(s')$, where we admit $s,s'$ from the bigger set $\se_n(1)\supset\se^0_n(1)$ to make the induction go through. For $s=\overline 0=s'$, both sides of our equivalence amount to~$t\leq_Z t'$. Now consider $s=\overline 0$ and $s'=\vartheta_jr'$. In view of $S(\pi_n(s,t))=S(t)\leq 0<j+1$ and $\pi_n(s',t')=\vartheta_{j+1}\pi_n(r',t')$, the left side of the desired equivalence is equivalent to
\begin{equation*}
\pi_n(s,t)\leq_{\se(1)} k_{j+1}(\pi_n(r',t'))=\pi_n(k_j(r'),t').
\end{equation*}
Inductively, this is equivalent to the conjunction of $s\leq_{\se(1)}k_j(r')$ and $t\leq_Z t'$. We can conclude since both $s\leq_{\se(1)}s'$ and $s\leq_{\se(1)}k_j(r')$ are automatic for $s=\overline 0$. If we have $s=\vartheta_ir$ and $s'=\overline 0$, then both sides of the desired equivalence fail. To see this, note that we have
\begin{equation*}
S(\pi_n(s,t))=S(\vartheta_{i+1}\pi_n(r,t))=i+1>0\geq S(t')=S(\pi_n(s',t'))
\end{equation*}
and $S(s)=i>-1=S(s')$. At the same time, a straightforward induction shows that $s_0\leq_{\se(1)}s_1$ entails~$S(s_0)\leq S(s_1)$. Finally, consider $s=\vartheta_ir$ and $s'=\vartheta_jr'$. The left side of the desired equivalence can hold for two reasons: First, assume that we have $\pi_n(s,t)\leq_{\se(1)}k_{j+1}(\pi_n(r',t'))$. As before, we get $s\leq_{\se(1)}k_j(r')$ and $t\leq_Z t'$. The former entails $s\leq_{\se(1)}s'$, as needed for the right side. Now assume that the left side holds because we have $i=j$ and $\pi_n(r,t)\leq_{\se(1)}\pi_n(r',t')$. Inductively we get $r\leq_{\se(1)}r'$ and $t\leq_Z t'$, and the former entails $s\leq_{\se(1)}s'$. By reading the given argument backwards, one obtains the implication from right to left. To show that our Kruskal fixed point is initial, we apply the criterion from~\cite[Theorem~3.5]{freund-kruskal-gap}. As preparation, one shows that any~$r\in\se_{n+1}(1)$ lies in the range of~$\pi_n$, using induction over~$r$. To conclude that $\kappa_n$ is surjective, it suffices to observe that $\vartheta_0r\in\se^0_{n+1}(1)$ requires $S(r)\leq 1$, so that $r=\pi_n(s,t)$ forces $S(s)\leq 0$ and hence
\begin{equation*}
(s,t)\in\se^0_n(1)\times\se^0_{n+1}(1)\subseteq W_n(\se^0_{n+1}(1)).
\end{equation*}
Finally, we need to verify that
\begin{equation*}
r\in\supp_Z(\sigma)\quad\To\quad h(r)<h(\kappa_n(\sigma))
\end{equation*} 
holds for all $\sigma\in W_n(Z)=W_n(\se^0_{n+1}(1))$. For $\sigma=0$ we have $\supp_Z(\sigma)=\emptyset$, so that the condition is void. For $\sigma=(s,t)$ we have $\supp_Z(\sigma)=\{t\}$, which means that the claim amounts to $h(t)<h(\kappa_n(s,t))$. This reduces to $h(t)\leq h(\pi_n(s,t))$, which is readily established by induction over~$s$.
\end{proof}

We have just seen a reconstruction of~$\se^0_n(1)$ via iterated Kruskal fixed points. In the following, we show that the parallel construction for Bachmann-Howard fixed points yields the orders~$OT_n[0]$ from~\cite[Section~5]{MRW-linear}. Let us first recall the latter.

\begin{definition}\label{def:ot}
For each $n\in\mathbb N$ we construct a set $\ot_n$ of terms. Simultaneously, each term $s\in\ot_n$ is associated with a number $S(s)\in\{-1,\dots,n-1\}$ and finite subsets $K_i(s)\subseteq\ot_n$ for $i\geq -1$, according to the following clauses:
\begin{enumerate}[label=(\roman*)]
\item We have a term $\overline 0\in\ot_n$ with $S(\overline 0)=-1$ and $K_i(\overline 0)=\emptyset$.
\item Given $s,t\in\ot_n$ and $i<n$ with $i\geq\max\{S(s)-1,S(t),0\}$ and $K_i(s)=\emptyset$, we add a term $\theta_ist\in\ot_n$ with $S(\theta_ist)=i$ and
\begin{equation*}
K_j(\theta_ist)=\begin{cases}
\{\theta_ist\} & \text{if $i\leq j$},\\
K_j(s)\cup K_j(t) & \text{otherwise}.
\end{cases}
\end{equation*}
\end{enumerate}
To define a binary relation~$<_{\ot}$ on $\ot_n$, we declare that $r<_{\ot}r'$ holds if, and only if, one of the following clauses applies:
\begin{enumerate}[label=(\roman*')]
\item We have $r=\overline 0$ and $r'\neq\overline 0$.
\item We have $r=\theta_ist$ and $r'=\theta_js't'$ with $i<j$.
\item We have $r=\theta_ist$ and $r'=\theta_is't'$, and one of the following holds:
\begin{itemize}[label={--}]
\item We have $s<_{\ot}s'$ and $t<_{\ot}r'=\theta_is't'$.
\item We have $s=s'$ and $t<_{\ot}t'$.
\item We have $\theta_ist=r\leq_{\ot}t'$ (i.\,e. $r<_{\ot}t'$ or $r=t'$ as terms).
\end{itemize}
\end{enumerate}
Finally, we set $\ot^0_n=\{s\in\ot_n\,|\,S(s)\leq 0\}$. We will also write $<_{\ot}$ for the restriction of the given relation to this set.
\end{definition}

The reader may have noticed that our definition of~$<_{\ot}$ looks somewhat different from the one in~\cite[Definition~37]{MRW-linear}. This has the following context: First, the cited reference also defines~$<_{\ot}$ on a larger set that contains terms $\theta_ist$ with~$K_i(s)\neq\emptyset$. This leads to additional conditions that are void in our case. Secondly, we have declared that $\theta_ist\leq_{\ot}t'$ entails $\theta_ist<_{\ot}\theta_is't'$. In~\cite{MRW-linear}, this implication has $s'<_{\ot}s$ as an additional assumption. At the same time,~\cite[Lemma~23]{MRW-linear} shows that $t'<\theta_is't'$ does always hold. As it is transitive, the order from~\cite{MRW-linear} will thus satisfy our stronger implication, so that the two definitions coincide. On a different note,~\cite[Definition~40]{MRW-linear} declares that we have
\begin{equation*}
OT_n[0]\stackrel{?}{=}\{s\in\ot_n\,|\,s<_{\ot}\theta_0\overline 0\overline 0\}.
\end{equation*}
It is easy to see that this yields $OT_n[0]=\{\overline 0\}$, which contradicts~\cite[Corollary~6]{MRW-linear}. Presumably, this is a typo that should be corrected into
\begin{equation*}
OT_n[0]=\{s\in\ot_n\,|\,s<_{\ot}\theta_1\overline 0\overline 0\}.
\end{equation*}
Strictly speaking, this is only explained when we have $n>1$, so that $\theta_1\overline 0\overline 0$ is available. For the corrected definition one can show $OT_n[0]=\ot^0_n$ by a straightforward induction over terms. The following was left implicit in~\cite{MRW-linear}:

\begin{lemma}\label{lem:ot-linear}
The relation $<_{\ot}$ is a linear order on~$\ot_n$, for each $n\in\mathbb N$.
\end{lemma}
\begin{proof}
Define $h:\ot_n\to\mathbb N$ recursively by $h(\overline 0)=0$ and $h(\theta_ist)=h(s)+h(t)+1$. A straightforward induction over $h(r)+h(s)+h(t)$ shows that the conjunction of $r<_{\ot}s$ and $s<_{\ot}t$ implies $r<_{\ot}t$. We can now establish $r\not<_{\ot}r$ by induction over~$h(r)$. This is readily reduced to $r=\theta_ist\not\leq_{\ot}t$. In view of $\theta_ist\in\ot_n$ we must have~$S(t)\leq i$, which leaves a term of the form $t=\theta_is't'$ as the only interesting case (with the same~$i$ as in $r=\theta_ist$). We trivially have $\theta_is't'\leq_{\ot}t$, which provides the inequality $t=\theta_is't'<_{\ot}\theta_ist=r$. Now if $r\leq_{\ot}t$ was true, then transitivity would yield $t<_{\ot}t$, which contradicts the induction hypothesis. Finally, trichotomy between~$s$ and $t$ follows by a straightforward induction over~$h(s)+h(t)$.
\end{proof}

Parallel to the above, our aim is to show that $\ot^0_{n+1}$ is the initial Bachmann-Howard fixed point of the following $\lo$-dilators~$D_n$ over the empty order~$0$.

\begin{definition}\label{def:D_n}
For each~$n\in\mathbb N$ and any linear order~$X$ we put
\begin{equation*}
D_n(X)=1+\ot^0_n\times X=\{0\}\cup\{(s,x)\,|\,s\in\ot^0_n\text{ and }x\in X\}.
\end{equation*}
To define a linear order on~$D_n(X)$, we declare that $0$ is the smallest element, and that we have
\begin{equation*}
(s,x)<_{D_n(X)}(s',x')\quad\Leftrightarrow\quad s<_{\ot}s'\text{ or }(s=s'\text{ and }x<_X x').
\end{equation*}
For any embedding $f:X\to Y$, we define a function $D_n(f):D_n(X)\to D_n(Y)$ by $D_n(f)(0)=0$ and $D_n(f)(s,x)=(s,f(x))$. Also, let $\supp_X:D_n(X)\to[X]^{<\omega}$ be given by $\supp_X(0)=\emptyset$ and $\supp_X(s,x)=\{x\}$.
\end{definition}

Let us construct the functions that Definition~\ref{def:BH-fp} requires:

\begin{definition}
Let the map $\ot_n\ni s\mapsto s^+\in\ot_{n+1}$ be given by the recursive clauses $\overline 0^+=\overline 0$ and $(\theta_ist)^+=\theta_{i+1}s^+t^+$. Now define $\vartheta_n:D_n(\ot^0_{n+1})\to\ot^0_{n+1}$ by
\begin{equation*}
\vartheta_n(0)=\overline 0\quad\text{and}\quad\vartheta_n(s,t)=\theta_0s^+t.
\end{equation*}
Let us also agree to write $\iota_n:0\to\ot^0_{n+1}$ for the empty function.
\end{definition}

To justify the definition of $s^+\in\ot_{n+1}$, one should simultaneously check
\begin{equation*}
K_{i+1}(s^+)=\{t^+\,|\,t\in K_i(s)\}\quad\text{and}\quad S(s^+)=\begin{cases}
S(s)=-1 & \text{if $s=\overline 0$},\\
S(s)+1 & \text{otherwise}.
\end{cases}
\end{equation*}
In particular we get $K_0(s^+)=\emptyset$, as an easy induction over~$s$ yields~$K_{-1}(s)=\emptyset$. For $s\in\ot^0_n$ we also obtain $S(s^+)\leq 1$, so that we indeed have $\vartheta_n(s,t)=\theta_0s^+t\in\ot^0_{n+1}$ for $(s,t)\in D_n(\ot^0_{n+1})$. We can now establish the promised result:

\begin{theorem}\label{thm:ot_n-fixed-point}
The tuple $(\ot^0_{n+1},\iota_n,\vartheta_n)$ is an initial Bachmann-Howard fixed point of~$D_n$ over~$0$, for each number~$n\in\mathbb N$.
\end{theorem}
\begin{proof}
As preparation one checks that $s<_{\ot}t$ implies $s^+<_{\ot}t^+$, by a straightforward induction over~$h(s)+h(t)$ (cf.~\cite[Lemma~24]{MRW-linear}). Abbreviating $Z=\ot^0_{n+1}$, we now verify the two conditions from clause~(iii) of Definition~\ref{def:BH-fp}. First, we need to show that $z\in\supp_Z(\sigma)$ implies $z<_Z\vartheta_n(\sigma)$, for any $\sigma\in D_n(Z)$. In the case of $\sigma=0$ we have $\supp_Z(\sigma)=\emptyset$, which means that the condition is void. For $\sigma=(s,t)$ we have $\supp_Z(s,t)=\{t\}$, so that the claim amounts to
\begin{equation*}
t<_Z\vartheta_n(s,t)=\theta_0s^+t.
\end{equation*}
In the proof of Lemma~\ref{lem:ot-linear} we have shown $\theta_0s^+t\not\leq_{\ot}t$, which suffices due to trichotomy (in view of $S(t)\leq 0$ it is easy to give a direct argument as well). The second condition from clause~(iii) of Definition~\ref{def:BH-fp} does essentially amount to
\begin{equation*}
(s_0,t_0)<_{D_n(Z)}(s_1,t_1)\text{ and }t_0<_Z\theta_0s_1^+t_1\quad\To\quad\theta_0s_0^+t_0<_Z\theta_0s_1^+t_1.
\end{equation*}
Assuming that the left side holds, we have $s_0^+=s_1^+$ and $t_0<_{\ot}t_1$, or $s_0^+<_{\ot}s_1^+$ and $t_0<_Z\theta_0s_1^+t_1$, using the fact that we have shown as preparation. In either case, the right side follows by clause~(iii') of Definition~\ref{def:ot}. To show that our fixed point is initial, we use the criterion from Theorem~\ref{thm:initial-characterize}. Let us first observe
\begin{equation*}
h(t)<h(s^+)+h(t)+1=h(\theta_0s^+t)=h(\vartheta_n(s,t)),
\end{equation*}
for $h:\ot^0_{n+1}\to\mathbb N$ as in the proof of Lemma~\ref{lem:ot-linear}. It remains to show that the function $\vartheta_n:D_n(\ot^0_{n+1})\to\ot^0_{n+1}$ is surjective. An easy induction shows that any element $s\in\ot_{n+1}$ with $K_0(s)=\emptyset$ can be written as $s=s_0^+$ for some $s_0\in\ot_n$. For $S(s)\leq 1$ we get $s_0\in\ot^0_n$, and it is straightforward to conclude. 
\end{proof}

Finally, we use our general approach to re-derive two known results. The following was first shown as a consequence of~\cite[Lemma~25]{MRW-linear}.

\begin{corollary}
For each $n\in\mathbb N$ we have a linearization $f_n:\ot^0_n\to\se^0_n(1)$ of the partial order~$\se^0_n(1)$ (sequences with gap condition) by the linear order $\ot^0_n$ (binary collapsing functions, cf.~\cite[Section~5]{MRW-linear}).
\end{corollary}
\begin{proof}
We construct~$f_n$ by recursion over~$n\in\mathbb N$. In the base case, let $f_0$ be the identity on~$\ot^0_0=\{\overline 0\}=\se^0_0(1)$. Given $f_n$, we get a linearization $\nu:D_n\To W_n$ by setting $\nu_X(0)=0$ and $\nu_X(s,x)=(f_n(s),x)$. It is easy to see that~$W_n$ is flat (cf.~Definition~\ref{def:dil-flat}). Theorems~\ref{thm:quasi-embeddings-derivatives} and~\ref{thm:linearization-preserved} yield a linearization $\nu^+:\vartheta D_n\To\T W_n$. From Theorems~\ref{thm:se_n-fixed-point-2} and~\ref{thm:ot_n-fixed-point} we know that $\T W_n(0)$ and $\vartheta D_n(0)$ coincide with $\se^0_{n+1}(1)$ and $\ot^0_{n+1}$, respectively (up to isomorphism). Hence setting
\begin{equation*}
f_{n+1}=\nu^+_0:\ot^0_{n+1}=\vartheta D_n(0)\to\T W_n(0)=\se^0_{n+1}(1)
\end{equation*} 
completes the recursive construction.
\end{proof}

The following result was first shown in~\cite[Corollaries~6 and~7]{MRW-linear}. The proof in the cited paper goes via an addition free variant of the Veblen functions. In the author's opinion, our general approach leads to a much more transparent argument:

\begin{corollary}
For each $n\geq 1$, the linear order~$\ot^0_n$ has order type $\omega_{2n-1}$, which means that it realizes the maximal order type of the partial order~$\se^0_n(1)$.
\end{corollary}
\begin{proof}
Recall that we have $\omega^\alpha_0=\alpha$ and $\omega^\alpha_{n+1}=\omega^{\omega^\alpha_n}$. We will show
\begin{equation*}
\ot^0_n\cong 1+\omega^0_{2n}
\end{equation*}
by induction over~$n\in\mathbb N$. This entails the first claim of the corollary, as $n\geq 1$ yields
\begin{equation*}
1+\omega^0_{2n}=\omega^0_{2n}=\omega^1_{2n-1}=\omega_{2n-1}.
\end{equation*}
In the base case of our induction, it suffices to recall $\ot^0_0=\{\overline 0\}$ and $\omega^0_0=0$. The induction step relies on~\cite[Theorem~2.2]{freund-predicative-collapsing}. To state the latter, consider an arbitrary ordinal~$\alpha$. Analogous to Definition~\ref{def:D_n}, we get an $\lo$-dilator~$D^\alpha$ with
\begin{equation*}
D^\alpha(X)=1+(1+\alpha)\times X.
\end{equation*}
According to the cited theorem, we have $\vartheta D^\alpha(0)\cong\omega_2^\alpha$. From Theorem~\ref{thm:ot_n-fixed-point} we know $\ot^0_{n+1}\cong\vartheta D_n(0)$ with $D_n(X)=1+\ot^0_n\times X$. Now the induction hypothesis provides a natural isomorphism $D_n\cong D^\alpha$ for $\alpha=\omega^0_{2n}$. In view of Lemma~\ref{lem:deriv-it-unique} (and the uniqueness of initial fixed points), we obtain
\begin{equation*}
\ot^0_{n+1}\cong\vartheta D_n(0)\cong\vartheta D^{\omega_{2n}^0}(0)\cong\omega_2^{\omega_{2n}^0}=\omega_{2(n+1)}^0=1+\omega_{2(n+1)}^0,
\end{equation*}
as needed to complete the induction step. Finally, let us give references for the fact that $\se^0_n(1)$ has maximal order type~$\omega_{2n-1}$, for $n\geq 1$. In the paragraph after Lemma~\ref{lem:se-partial-order} we have observed that our order~$\se^0_n(1)$ coincides with the order~$(T_n[0],\trianglelefteq)$ from~\cite[Section~2.2.3]{MRW-linear}. The cited reference shows that this order is isomorphic to a certain collection $\overline{\mathbb S}_n[0]$ of sequences with the strong gap condition. According to~\cite[Lemma~3 and Corollary~2]{MRW-linear} (based on work by Sch\"utte and Simpson~\cite{schuette-simpson}), that collection has maximal order type~$\omega_{2n-1}$, for $n\geq 1$.
\end{proof}

\bibliographystyle{amsplain}
\bibliography{Bachmann-Howard-derivatives}

\end{document}